\documentclass{article}
\usepackage[utf8]{inputenc}
\usepackage{tikz}
\usepackage{mathdots}
\usepackage{tabularx}
\usepackage{extarrows}
\usepackage{enumerate}
\usepackage{caption}  
\usepackage{booktabs}
\usetikzlibrary{fadings}
\usetikzlibrary{patterns}
\usepackage{float}
\usetikzlibrary{shadows.blur}
\usetikzlibrary{shapes}

\title
{Magnetic ground states and the conformal class of a surface}

\author{Bruno Colbois\footnote{Universit\'e de Neuch\^atel\,, Institut de Math\'emathiques\,, Rue Emile-Argand 11\,, CH2000 Neuch\^atel\,, Switzerland. Email:\, bruno.colbois@unine.ch}, Luigi Provenzano\footnote{Sapienza Universit\`a di Roma\,, Dipartimento di Scienze di Base e Applicate per l'Ingegneria\,, Via Scarpa 16\,, 00161 Roma\,, Italy. Email:\, luigi.provenzano@uniroma1.it}, and Alessandro Savo\footnote{Sapienza Universit\`a di Roma\,, Dipartimento di Scienze di Base e Applicate per l'Ingegneria\,, Via Scarpa 16\,, 00161 Roma\,, Italy. Email:\, alessandro.savo@uniroma1.it}}
\date{\today}

\usepackage{makeidx}
\usepackage{amssymb , amsmath, amsthm}
\usepackage{graphicx}
\usepackage{xcolor}
\newtheorem{defi}{Definition} 
\newtheorem{thm}[defi]{Theorem}

\newtheorem{rem}[defi]{Remark}
 \newtheorem{prop}[defi]{Proposition}
\newtheorem{lemme}[defi]{Lemma}
\newtheorem{cor}[defi]{Corollary}
 
\newcommand{\twosystem}[2]{\left\{\begin{aligned} &#1\\ &#2\end{aligned}\right.}

\newcommand{\nero}{\smallskip$\bullet\quad$\rm}

\newcommand{\due}[2]{#1&#2\\}

\newcommand{\matrice}{\begin{pmatrix}}
\newcommand{\ok}{\end{pmatrix}}
\newcommand{\twomatrix}[4]{\matrice\due {#1}{#2}\due{#3}{#4}\ok}

\newcommand{\derive}[2]{\dfrac{\bd #1}{\bd#2}}

\newcommand{\Tre}[3]{\begin{pmatrix} #1\\#2\\#3\end{pmatrix}}

\newcommand{\threearray}[3]{\begin{aligned}&{#1}\\&{#2}\\&{#3}\end{aligned}}

\newcommand{\scal}[2]{\langle{#1},{#2}\rangle}

\newcommand{\abs}[1]{\lvert{#1}\rvert}
\newcommand{\norm}[1]{\lVert{#1}\rVert}

\newcommand{\reals}{{\mathbb R}}

\newcommand{\real}[1]{{\mathbb R}^{#1}}

\newcommand{\bd}{\partial}

\newcommand{\C}{\mathbb C}

\begin{document}

\maketitle

\noindent
{\bf Abstract.} On a closed, orientable Riemannian surface $\Sigma_g$ of arbitrary genus $g\geq 1$ and Riemannian metric $h$ we study the magnetic Laplacian with magnetic potential given by a harmonic $1$-form $A$. Its lowest eigenvalue (magnetic ground state energy) is positive, unless $A$ represents an integral cohomology class. We isolate a countable set of ground state energies which we call {\it ground state spectrum} of the metric $h$. The main result of the paper is to show that the ground state spectrum  determines the volume and the conformal class of the metric $h$. In particular, hyperbolic metrics are distinguished by their ground state spectrum.

We also compute the magnetic spectrum of flat tori and introduce some magnetic spectral invariants of $(\Sigma_g,h)$ which are conformal by definition and involve the geometry of what we call the Jacobian torus of $(\Sigma_g,h)$ (in Algebraic Geometry, the Jacobian variety of a Riemann surface).

\vspace{11pt}


\vspace{6pt}

\section{Introduction}

In this paper we will study the ground state energy of the magnetic Laplacian with zero magnetic field on $2$-dimensional compact surfaces with empty boundary. In this short introduction we define the magnetic Laplacian and introduce some of its fundamental properties. The main results will be discussed in Section \ref{sec:results}. 

 Let $(M,h)$ be a compact Riemannian manifold without boundary; note that $h$ denotes the metric (the notation $g$ is reserved for the genus of a surface, considered below). 
 
\smallskip
 
 A {\it magnetic potential} is by definition a smooth real $1$-form $A$ on $M$. It defines a {\it magnetic gradient}  $\nabla^A$ acting on $C^{\infty}(M,\C)$  by the formula
 $$
 \nabla^Au=\nabla u-iuA^{\sharp}
 $$
 where $\nabla$ denotes the gradient in the metric $h$ and $A^{\sharp}$ is the dual vector field of $A$ ({\it vector potential}). Dually, one can define the {\it magnetic differential}
 of $u$ as the complex $1$-form
$
d^Au=du-iuA.
$

The \emph{magnetic Laplacian} $\Delta_A$ is given by $(\nabla^A)^*\nabla^A$ and may be written
\begin{equation}\label{ml}
\Delta_Au=\Delta u+2i\langle A,du\rangle_h+(\vert A\vert_h^2-i\delta A)u
\end{equation}
where $\delta=d^{\star}$ denotes the codifferential and $\Delta$ the usual Laplacian; our sign convention is that, on $\real n$: $\Delta u=-\sum_k\frac{\bd^2u}{\bd x_k^2}$. Equivalently, $\Delta_A=\delta^Ad^A$, where $\delta^A$ is the magnetic codifferential, acting on the $1$-form $\omega$ by $\delta^A\omega=\delta\omega+i\langle{A},{\omega}\rangle_h$.

\smallskip

 Since $\Delta_A$ defines a compact self-adjoint operator with dense domain in $L^2(M,dv_h)$, we have a non-negative, discrete spectrum 
 $$
0\leq \lambda_1(M,h,A)\leq  \lambda_2(M,h,A)\leq\dots\leq \lambda_k(M,h,A)\leq\dots
 $$ 
 diverging to $+\infty$.
 The \emph{magnetic ground state energy} is the first eigenvalue $\lambda_1(M,h,A)$ of $\Delta_A$; it has the usual variational characterization
\begin{equation}\label{minmax}
 \lambda_1(M,h,A)=\min_{0\ne u\in C^{\infty}(M,\C)}R_A(u)
 \end{equation}
 where $R_A$ is the Rayleigh quotient associated with $\Delta_A$:
 $$
 R_A(u)=\frac{\int_M\vert \nabla^Au\vert_h^2 dv_h}{\int_M\vert u\vert^2dv_h}.
 $$

\medskip

One sees immediately that $\Delta_A$ reduces to the usual Laplace-Beltrami operator when $A=0$, giving rise to the same spectrum:
$$
\lambda_k(M,h,0)= \lambda_k(M,h)
$$
where on the right we have the $k$-th eigenvalue of the Laplace-Beltrami operator. Note that we number the eigenvalues so that $\lambda_1(M,h)=0$, hence, in this paper,  the lowest positive eigenvalue of the Laplace-Beltrami operator is $\lambda_2(M,h)$. 

\smallskip

A fundamental feature of the magnetic Laplacian is {\it gauge invariance}, expressed in the following lemma. Recall that the flux of a $1$-form $A$ on the closed curve $c$ is defined as $\frac{1}{2\pi}\oint_cA$. 

\nero We say that the closed $1$-form $A$ is {\it integral} if the flux of $A$ around every loop in $M$ is an integer. 

\smallskip

Note that exact $1$-forms have zero flux around any loop, hence they are trivially integral. 

\begin{lemme} Assume that $A$ and $A'$ are closed potential $1$-forms such that $A-A'$ is integral. Then 
$\lambda_k(M,h,A)=\lambda_k(M,h,A')$ for all $k\in\mathbb N$.
\end{lemme}
In fact, if $A-A'$ is exact: $A-A'=d\phi$, the operators $\Delta_A$ and $\Delta_A'$ are unitarily equivalent because of the formula
$$
\Delta_Ae^{i\phi}=e^{i\phi}\Delta_{A'}
$$
and then have the same spectrum. If more generally $A-A'$ is integral the above holds with 
$\phi(x)=\int_{x_0}^x(A-A')\doteq\int_{\gamma}(A-A')$, where $x_0$ is any fixed point and the integral is taken along any curve $\gamma$ joining $x_0$ and $x$. Note that $\phi(x)$ is multi-valued: if $\gamma$ and $\tilde\gamma$ are curves joining $x_0$ and $x$  then $\int_{\gamma}(A-A')-\int_{\tilde\gamma}(A-A')\in 2\pi\mathbb Z$ by the integrality assumption, hence the function $e^{i\phi}$ is well-defined and smooth. 

\smallskip

Our main attention is given to $\lambda_1(M,h,A)$, the {\it ground state energy} of $(M,h)$ for the potential $A$: it is important to know when it is strictly positive. The following fact is due to Shigekawa (for a proof, see \cite[Appendix 5]{CEIS2022}).

\begin{thm}\label{shi} One has $\lambda_1(M,h,A)=0$ if and only if $A$ is closed and integral (that is, A has integral flux around every loop in $M$). 
\end{thm} 
In this paper we will consider {\it closed} potentials $A$, so that the magnetic field $B$ vanishes: $B\doteq dA=0$. According to Theorem \ref{shi}, the ground state energy could be strictly positive, precisely when the potential $A$ has non-integral flux around some loop in $M$.
In the literature, this corresponds to a phenomenon in quantum mechanics called {\it Aharonov-Bohm effect}, observed in experimental physics: hence the condition $dA=0$ is far from being trivial. This condition is very interesting also from the spectral geometric point of view, as we shall see very soon.

 \smallskip
 
 Let ${\rm Har}(M,h)$ be the vector space of $h$-harmonic $1$-forms on $(M,h)$; by the Hodge-de Rham theorem, it is isomorphic to the first de Rham cohomology space $H^1(M,\mathbb R)$. Observe that, if $A$ is closed, we have by gauge invariance that $\lambda_1(M,h,A)=\lambda_1(M,h,\tilde A)$, where $\tilde A$ is the unique harmonic representative of $A$. Hence, we can restrict our attention to {\it harmonic} potential forms, and so we can view the ground state energy as a function
$$
\lambda_1(M,h,\cdot):{\rm Har}(M,h)\to [0,\infty)
$$
taking $A$ to $\lambda_1(M,h,A)$. A simple but important result of this paper is the computation of the Hessian of this function at $A=0$ (see Theorem \ref{thm:asymp}).

\smallskip

Next,  let $\mathcal L^*$ be the lattice in ${\rm Har}(M,h)$ formed by the integral $1$-forms; as an abelian group, it is isomorphic to $H^1(M,\mathbb Z)$. Note that Theorem \ref{shi} says precisely  that 

\nero $\lambda_1(M,h,A)=0$ if and only if $A\in\mathcal L^{\star}$.

\smallskip

By gauge invariance, we conclude that the ground state energy is really a function on the quotient:
$$
\lambda_1(M,h,\cdot):{\rm Har}(M,h)/{\mathcal L^{\star}}\to [0,\infty).
$$

\nero Suitably renormalized, the quotient torus ${\rm Har}(M,h)/{\mathcal L^{\star}}$ will be called the {\it Jacobian torus} of $(M,h)$ (see Subsection \ref{sub:spec_inv}).
When $M$ is a Riemannian surface, it is a conformal invariant, known in Algebraic Geometry  as the Jacobian variety of $(M,h)$.

\smallskip

We now recall a result from \cite{CEIS2022} which partly inspired the present paper. We endow ${\rm Har}(M,h)$ with the $L^2$-inner product of forms, and define the distance of $A\in {\rm Har}(M,h)$ to the lattice $\mathcal L^*$ as follows:
 \begin{equation}\label{dist_har}
 d_{h}(A,\mathcal L^*)^2=\min\{\Vert \omega-A \Vert^2_{L^2(h)};\ \omega \in \mathcal L^*\}.
 \end{equation}
Then we have the upper bound \cite[Theorem 3]{CEIS2022}:
 \begin{equation} \label{ineg 1}
 \lambda_1(M,h,A)\le \frac{1}{{\rm vol}(M,h)}  d_h(A,\mathcal L^*)^2.
 \end{equation}
It also follows from \cite[Theorem 5]{CEIS2022} that when $M$ is a two-dimensional torus, we have equality in (\ref{ineg 1}) if and only if $M$ is flat. 

\subsection{Preliminary comments}

The spectrum of the magnetic Laplacian has received great attention lately, also from the 
purely mathematical point of view, on which we focus here. In the paper \cite{CEIS2022} a series of upper bounds have been proved, some of them generalizing the results and the methods employed in the classical non-magnetic case, i.e. the Laplace-Beltrami operator. For a general study of the magnetic Laplacian see for example \cite{FoHe2010}. A case which is widely studied is that of plane domains endowed with a constant magnetic field, see for instance  \cite{raymond_noel,CLPS_1,CLPS_rev,erdos_0,FoHe2010,frank_1,loto_kach_iso_robin,KaLo2024} and the references therein. Aharonov-Bohm magnetic potentials play a role in the geometry of spectral minimal partitions of plane domains \cite{terracini_helffer}; in two dimensions, isoperimetric inequalities for the lowest eigenvalue have been obtained in \cite{CPS_AB}, and Steklov eigenvalues estimates in the spirit of Fraser-Schoen have been studied for free boundary surfaces in the unit ball of $\real 3$ in \cite{PS_annuli}. Recently, magnetic spectral asymptotics for general magnetic potentials have been studied in \cite{Siffert} and, finally, we mention also \cite{gittins_hodge}  where the authors study the spectrum of the magnetic Hodge-Laplace operator acting on forms.

\smallskip

In  the next section we state our main results. Briefly, we first show that the main inequality proven in \cite{CEIS2022} is sharp, asymptotically, as the potential form $A$ tends to zero. We then focus on two dimensional closed, orientable Riemannian surfaces of arbitrary genus $g$ endowed with a closed magnetic potential. This {\it ``weak magnetic potential asymptotics''} will be used to isolate a countable family of ground state energies, called {\it ground state spectrum}, the knowledge of which, as shown in our main result, determines the volume and the conformal class of the surface; in particular, ground state-isospectral (in our sense) hyperbolic surfaces are isometric, a result which does not hold for the classical Laplace-Beltrami spectrum.

\smallskip

A concrete case of study will be flat tori, for which we compute in detail the magnetic spectrum and discuss optimization of the magnetic ground state energy under volume constraint.  An important tool in our two-dimensional analysis is the conformal invariance of the magnetic energy, which holds in our Aharonov-Bohm magnetic situation. Thanks to this invariance, the above optimization extends to any surface of genus one. 

\smallskip

We finally introduce some spectral invariants in the arbitrary genus $g$, which are conformal by definition and involve the geometry of the so-called Jacobian torus (or Jacobian variety) of the surface, which is (magnetic) spectrally determined and naturally arises from the lattice of integral $1$-forms on the surface. The Jacobian variety plays a crucial role in algebraic geometry \cite{FK}; 
see \cite{buser21,BuSa} for a more differential geometric approach. See also the introduction of the Ph.D. thesis of B. Mützel \cite{Mutzel}.

\section{The results of the paper}\label{sec:results}

\subsection{Asymptotics of the ground state energy for weak magnetic potentials}

Our first result is rather general and applies to any compact Riemannian manifold, complementing \cite[Theorem 2]{CEIS2022}. We determine the  asymptotics of $\lambda_1(M,h,A)$ as $\norm{A}_{L^2(h)}\to 0$, in the following precise sense.

 \begin{thm}\label{thm:asymp} Let $(M,h)$ be a compact Riemannian manifold and let $A$ be a co-closed  potential. Then, there exist constants $\rho>0$ and $C>0$ such that for all $r\in (0,\rho)$:
 
 \begin{equation}\label{asymptotics}
 \frac{\Vert A\Vert_{L^2(h)}^2}{{{\rm vol}(M,h)}}r^2-Cr^4\le     \lambda_1(M,h,rA) \le \frac{\Vert A\Vert_{L^2(h)}^2}{{\rm vol}(M,h)}r^2.
 \end{equation}
 In particular, since for $r$ small $d_h(rA,\mathcal L^{\star})^2=r^2\norm{A}^2_{L^2(h)}$, it follows that
 \begin{equation}\label{limit}
     \lim_{r \to 0} \frac{{\rm vol}(M,h) \lambda_1(M,h,rA)}{d_h(rA,\mathcal L^{\star})^2} =1.
 \end{equation}
 \end{thm}

The limit \eqref{limit} in Theorem \ref{thm:asymp} shows that the estimate (\ref{ineg 1}) obtained in \cite{CEIS2022} is asymptotically exact as $\norm{A}_{L^2(h)}\to 0$. 

Note also that \eqref{asymptotics} can be re-written:
\begin{equation}\label{eltwonorm}
\lim_{r\to 0}\dfrac {{\rm vol}(M,h)\lambda_1(M,h,rA)}{r^2}=\norm{A}_{L^2(h)}^2
\end{equation}
showing that the knowledge of the lowest normalized eigenvalue ${\rm vol}(M,h)\lambda_1(M,h,rA)$ for $r$ in a small neighborhood of $0$ (or on a sequence $r_n\to 0$) determines the $L^2$-norm of $A$. In dimension $2$ it is known that the $L^2$-norm of a $1$-form is conformally invariant; thus, the limit in the left-hand side of \eqref{eltwonorm} is a conformal invariant as well. This phenomenon was first observed in \cite{PS_annuli} in the context of the magnetic Laplacian on Riemannian cylinders with magnetic Steklov boundary conditions; in fact Theorem \ref{thm:asymp} can be extended (in arbitrary dimension) to manifolds with boundary endowed with magnetic Neumann (or Steklov) boundary conditions.

\subsection{Main result: the magnetic ground state 
energies determine the conformal class of a surface}

We now focus on dimension $2$, and consider the (unique) compact, orientable, differentiable surface of genus $g\geq 1$, which we denote $\Sigma_g$. We fix the metric $h$ on $\Sigma_g$ and study the ground state energy
$$
\lambda_1(\Sigma_g,h,A)
$$
where $A$ is a closed $1$-form on $\Sigma_g$. We recall that this ground state energy is strictly positive, unless $A\in\mathcal L^{\star}$ (i.e. $A$ is integral), in which case it is zero.  Here is the main result of this paper.

\begin{thm}\label{mainthm} Assume that the metrics $h$ and $h_0$ on $\Sigma_g$ are such that:
\begin{equation}\label{iso}
\lambda_1(\Sigma_g,h,A)=\lambda_1(\Sigma_g,h_0,A)\quad\text{for all closed $1$-forms $A$ on $\Sigma_g$}.
\end{equation}
 Then $(\Sigma_g,h)$ and $(\Sigma_g,h_0)$ are conformal and have the same volume. 
\end{thm}

The proof uses Theorem \ref{thm:asymp} and a classical result in the theory of Riemann surfaces, the theorem of Torelli, to the effect that if two compact Riemann surfaces have the same Jacobian variety, they are conformal. 
Now let $g\geq 2$. Since there is only one metric of constant curvature $-1$ in any fixed conformal class, we have:

\begin{cor}\label{hyperbolic} Assume that $h$ and $h_0$ are hyperbolic metrics on $\Sigma_g$ which satisfy \eqref{iso}. Then $h$ and $h_0$ are isometric. 
\end{cor}

Note that the family of eigenvalues in \eqref{iso} needed to determine volume and conformal class consists only of lowest eigenvalues (ground state energies), but 
corresponding to an infinite family of potential $1$-forms, i.e., closed potentials. By gauge invariance, we can restrict this family to that of harmonic $1$-forms, which form a vector space of dimension $2g$. 
But indeed what we need is only the behavior of $\lambda_1$ as $A\to 0$, roughly speaking, the asymptotics for {\it low energy} states; this allows to  obtain the same results of Theorem \ref{mainthm}, by restricting \eqref{iso} to a countable family of magnetic potentials, hence, to a countable family of eigenvalues,  which we term {\it (magnetic) ground state spectrum} and which we define in the next subsection.

\subsection{The (magnetic) ground state spectrum of a surface}

We start by fixing  a canonical homology basis $(\chi_1,\dots,\chi_{2g})$ (see Subsection \ref{sub:can_bas} for the definition) and we let
$(\alpha_1,\dots,\alpha_{2g})$ be its dual basis of closed forms in $H^1(\Sigma_g,\mathbb Z)$. Note that $(\alpha_1,\dots,\alpha_{2g})$ depends only on the differentiable structure of $\Sigma_g$, and not on the metric.  Then we define the following family  of magnetic eigenvalues:
\begin{equation}\label{mgs}
\mu_{jk,n}(\Sigma_g,h)\doteq\lambda_1\left(\Sigma_g,h, \frac 1n(\alpha_j+\alpha_k)\right) \quad 1\leq j\leq k\leq 2g, \quad n\in\mathbb N.
\end{equation}
We call $\{\mu_{jk,n}(\Sigma_g,h)\}$ the {\it (magnetic) ground state spectrum of the metric $h$}.

\begin{thm}\label{thm_sequences} The volume and the conformal class of the metric $h$ on $\Sigma_g$ are determined by the ground state spectrum of $h$. 
\end{thm}


 In fact, by Theorem \ref{thm:asymp}, the data above are enough to reconstruct the entries $\Gamma_{jk}$ (rescaled by the volume) of the so-called {\it Gram matrix} (see Subsection \ref{sub:can_bas}), a $2g\times 2g$ matrix naturally associated to the conformal class of $h$, once  a canonical homology basis of $H_1(\Sigma_g,\mathbb Z)$ is fixed. That this matrix actually determines the conformal class of $h$ follows from a deep theorem in the theory of Riemann surfaces, due to Torelli \cite{torelli}, which we will use to conclude the proof. 

 \smallskip

 \begin{rem} Since what matters is only the behavior of the ground state when the potential form is small,  we can replace the potential forms $\frac1n(\alpha_j+\alpha_k)$ in \eqref{mgs} by $c_n(\alpha_j+\alpha_k)$
 where $\{c_n\}$ is any sequence of positive numbers converging to zero as $n\to\infty$; the conclusion of Theorem \ref{thm_sequences} still holds. Also, volume and conformal class are determined by the knowledge of
 $\mu_{jk,n}(\Sigma_g,h)$ when $n\geq N_0$, where $N_0$ is a fixed positive integer.
 \end{rem}

\smallskip

Corollary \ref{hyperbolic} takes the following form.

\begin{cor}\label{hyperbolictwo} Assume that $h$ and $h_0$ are hyperbolic metrics on $\Sigma_g$ with the same ground state spectrum. Then $h$ and $h_0$ are isometric. 
\end{cor}

\smallskip

In what follows, {\it isospectral} refers to the classical notion in spectral geometry, namely, same Laplace-Beltrami spectrum. According to Wolpert  \cite{Wo79} and Buser (see the introduction in \cite{Bu86}) Gelfand conjectured that two hyperbolic, isospectral surfaces are necessarily isometric.
This conjecture was first disproved by Vigneras in \cite{Vi78} for large genus and then by Buser \cite{Bu86} in genus 5. In particular, the Laplace-Beltrami spectrum does not determine the hyperbolic metric in the conformal class. Corollary \ref{hyperbolictwo} states that this happens in the magnetic sense, as defined above.

\smallskip

In genus one we have a similar phenomenon, in the sense that the flat metric is determined by the ground state spectrum.   In fact, this is true in any dimension, see Theorem \ref{isosp_ntori}. 

\begin{thm}\label{flat} Two $d$-dimensional flat tori with the same ground state spectrum are necessarily isometric.  
\end{thm}
It is well known that this is not true for the Laplace-Beltrami operator, because, besides Milnor's classical examples, there are flat tori in dimension  four (and greater) which are isospectral but not isometric.

\smallskip

Given Theorems \ref{thm_sequences} and \ref{flat} one could ask:

\nero {\it Is it true that two surfaces $(\Sigma_g,h)$ and $(\Sigma_g,h_0)$ which have the same ground state spectrum are actually isometric?}

\subsection{Some (conformal) spectral invariants}\label{sub:spec_inv} We just saw that the magnetic Laplacian associated to closed potentials has deep relations with the conformal structure of a surface. We introduce here some conformal invariants related to the magnetic ground state energy. To simplify the notation, through all the paper we shall denote by $|h|$ the volume of the surface $\Sigma_g$ for the metric $h$, namely $|h|:={{\rm vol}(\Sigma_g,h)}$.

For a given metric $h$ on $\Sigma_g$ and a closed $1$-form $A$, define:
\begin{equation}\label{firstci}
\Lambda_1(\Sigma_g,[h],A)=\sup_{h'\in[h]}\abs{h'} \lambda_1(\Sigma_g,h',A).
\end{equation}
We remark that the quantity $\abs{h} \lambda_1(\Sigma_g,h,A)$ is often called {\it normalized ground state energy} and it is easily seen to be scale invariant. The supremum is finite as we will see soon. By way of contrast, the infimum is zero, see Subsection \ref{sub:inf_0}. We then define 
\begin{equation}\label{secondci}
\Lambda_1(\Sigma_g,[h])=\sup_{A:dA=0}\Lambda_1(\Sigma_g,[h],A).
\end{equation}
Again note that the supremum can be taken over all {\it harmonic} potentials. Finally we define
$$
\Lambda_1(\Sigma_g)=\inf_{[h]}\Lambda_1(\Sigma_g,[h]).
$$
Note that this last invariant depends only on the genus, hence the topological structure of the surface. In the next section we will  study and compute the above invariants when the genus is one. 

\smallskip

{\bf The Jacobian torus.} All of the invariants in the previous section can be estimated from above  by the geometry of a flat torus canonically associated to the surface: the {\it Jacobian torus} of  $(\Sigma_g,h)$; in algebraic geometry, it is known as the {\it Jacobian variety} of the given (compact, orientable) Riemann surface, and plays an important role in the theory of Riemann surfaces. 

\smallskip

Let $(M,h)$ be a compact, $d$-dimensional Riemannian manifold with first Betti number $b_1(M)\geq 1$, and let $V\doteq{\rm Har}(M,h)$ be the vector space of $h$-harmonic $1$-forms endowed with the $L^2$-inner product:
$$
\scal{\omega}{\phi}_{L^2(h)}=\int_{M}\scal{\omega}{\phi}_hdv_h.
$$
As a vector space, $V$ is isomorphic to $\real {b_1(M)}$, by the Hodge theorem. We let $L^{\star}$ be the lattice in $V$ formed by those harmonic forms having {\it integral period} $\oint_{c}\omega$ around every loop $c$ in $M$. In the previous section, we used the lattice $\mathcal L^{\star}$ of forms with {\it integral flux} $\frac{1}{2\pi}\oint_{c}\omega$; hence, here is the relation between the two lattices:
$$
\mathcal L^{\star}=2\pi L^{\star}.
$$
For us, it will be convenient to work with $L^{\star}$ instead of $\mathcal L^{\star}$.  We define the {\it Jacobian torus} of $(M,h)$ as the flat torus
$$
{\rm Jac}(M,h)={\rm Har}(M,h)/L^{\star}.
$$
\nero When $M=\Sigma_g$ (the surface of genus $g$) the Jacobian torus is $2g$-dimensional; it is conformally invariant and has volume $1$ for all $g$ (see Theorem \ref{detone}).

\smallskip

We then focus on the surface $(\Sigma_g,h)$. We have the following estimates from above. Given a harmonic $1$-form $A$, we denote $P_A=\frac{1}{2\pi}A$ and let:
$$
d_h(P_A,L^{\star})^2=\min_{\omega\in L^{\star}}\norm{P_A-\omega}^2_{L^2(h)}.
$$
Note that $d_h(P_A,L^{\star})$ measures the minimum distance of the harmonic form $P_A$ to the integral lattice $L^{\star}$, for the $L^2$ metric. Consider the canonical projection
$$
\pi: {\rm Har}(\Sigma_g,h)\to{\rm{Jac}}(\Sigma_g,h)
$$
then, the image of $L^{\star}$ under $\pi$ is a distinguished point $\mathcal O\in \rm{Jac}(\Sigma_g,h)$. As $\pi$ is a local isometry, it is clear that 
$$
d_h(P_A,L^{\star})=d_h(\pi(P_A),\mathcal O).
$$
Next, we set
\begin{equation}\label{inradius_h}
\mathcal R_h(L^{\star})^2\doteq{\rm max}_{A\in {\rm Har}(h)} d_h(P_A,L^{\star})^2.
\end{equation}
$\mathcal R_h(L^{\star})$ is the {\it inradius} of the lattice $L^{\star}$; that is, the radius of the largest open ball not intersecting $L^{\star}$. It measures how dense is $L^{\star}$ inside ${\rm Har}(\Sigma_g,h)$. Note that, given any two points on a flat torus, there exists an isometry (translation) taking one to the other. This immediately implies:
$$
\mathcal R_h(L^{\star})={\rm diam}({\rm Jac}(\Sigma_g,h))
$$
the diameter of the associated Jacobian torus. The following estimate follows immediately from \cite[Theorem 3]{CEIS2022}.

\begin{thm}\label{cig} One has:
$$
\threearray
{\Lambda_1(\Sigma_g,[h],A)\leq 4\pi^2 d_h(P_A,L^{\star})^2}
{\Lambda_1(\Sigma_g,[h])\leq 4\pi^2\mathcal R_h(L^{\star})^2=4\pi^2
{\rm diam}({\rm Jac}(\Sigma_g,h))^2}
{\Lambda_1(\Sigma_g)\leq 4\pi^2\inf_{[h]}{\rm diam}({\rm Jac}(\Sigma_g,h))^2}.
$$
\end{thm}

{\bf Some remarks.} The above inequalities show that the spectral invariants defined in this section admit upper bounds in terms of the geometry of the Jacobian torus of the given conformal class.  In particular, the eigenvalue $\abs{h} \lambda_1(\Sigma_g,h,A)$ is uniformly bounded above for $h$ in any  given conformal class.  In Section \ref{bruno} we will make a number of remarks:

\begin{enumerate}[1)]

\item For arbitrary metrics, $\abs{h}\lambda_1(\Sigma_g,h,A)$ can be arbitrarily large; in fact, we prove in Theorem \ref{secondex} that  there exist a closed potential form $A$ and a sequence of metrics $h_n$ such that 
$$
\lim_{n\to\infty}\abs{h_n} \lambda_1(\Sigma_g,h_n,A)=+\infty.
$$

\item The normalized eigenvalue $\abs{h}\lambda_1(\Sigma_g,h,A)$ can be arbitrarily small, even in a given conformal class. In fact, we prove in Theorem \ref{thm_inf_0}:
$$
\inf_{h'\in [h]}\abs{h'}\lambda_1(\Sigma_g,h',A)=0.
$$

\item We will see in the next Subsection \ref{sub:intro:g1} that in genus one the supremum  in the definition of $\Lambda_1(\Sigma_g,[h])$ is actually a maximum, and is achieved by the unique (up to homothety) flat metric in the conformal class $[h]$.  Therefore it makes sense to ask the following question in genus $g\geq 2$:

\nero Is the hyperbolic metric in each class $[h]$ a maximiser for the invariant $\Lambda_1(\Sigma_g,[h])$?

\smallskip
To that end, in Theorem \ref{firstex} we prove the following universal upper bound: for any hyperbolic metric $h$ on $\Sigma_g$ and any harmonic potential $A$, one has:
$$
\abs{h}\lambda_1(\Sigma_g,h,A)\leq C(g)
$$
for a constant $C(g)$ depending only on $g$. It is now clear that the hyperbolic metric in the conformal class of the metric $h_n$ as in 1) cannot be a maximiser for large $n$. Hence, the question is more complicated.  

\item The invariant $\Lambda_1(\Sigma_g)$ is particularly interesting; in genus $1$ it equals $8\pi^2/(3\sqrt 3)$ (see Theorem \ref{intro:flat} below) and is achieved by the regular hexagonal tiling (i.e. flat equilateral torus) in the plane. Note that the optimal conformal class
corresponds to a special Jacobian torus, hence to a lattice in $\mathbb R^{2g}$ with specific geometric properties.
But the immediate question would be: is $\Lambda_1(\Sigma_g)$  positive for all $g$?
\end{enumerate}

\subsection{The magnetic spectrum of flat tori and the invariants in genus one}\label{sub:intro:g1} All the above invariants can actually be computed for $g=1$, that is when $\Sigma_1$ is the $2$-torus. In that case the inequalities of Theorem \ref{cig} are actual equalities. This follows from a simple application of the conformal invariance of the magnetic energy of a function $u$:
$
\int_{\Sigma}\abs{\nabla^Au}_h^2dv_h,
$
which gives easily that, for any metric $h$ on $\Sigma_1$:

\begin{thm}\label{flatisbest} Given any metric $h$ on the $2$-torus $\Sigma_1$ one has
$$
\abs{h}\lambda_1(\Sigma_1,h,A)\leq \abs{\hat h}\lambda_1(\Sigma_1,\hat h,A)
$$
where $\hat h$ is the unique (up to homotheties) flat torus in the conformal class of $h$. Equality holds iff $h$ and $\hat h$ are homothetic. In particular:
$$
\Lambda_1(\Sigma_1,[h],A)=\abs{\hat h}\lambda_1(\Sigma_1,\hat h,A).
$$
\end{thm}

\medskip

The magnetic spectrum of a flat $d$-dimensional torus with a harmonic potential form $A$ is then computed in Section \ref{sec:tori}; we already remarked that, by gauge invariance, it reduces to the classical Laplace-Beltrami spectrum precisely when $A$ has integral flux on every homology class, i.e. when $A\in 2\pi L^{\star}$; in that case $P_A=\frac{1}{2\pi}A$ has integral periods and consequently:
$$
d_h(P_A,L^{\star})=0.
$$
In general, the magnetic spectrum ${\rm spec}(\Sigma_1,\hat h,A)$ is (up to a multiplicative constant of $4\pi^2$) the set of 
the squares of all distances of $P_A$ to the lattice points in $L^{\star}$ (see \cite{CEIS2022}).
It is worth noticing that the family of eigenfunctions does not depend on $A$, and coincides, for all potentials $A$, with that of the usual Laplace-Beltrami operator; in particular, all eigenfunctions have constant modulus $\abs u=C$ on $\Sigma$. 

\medskip

Let then $(\Sigma_1,\hat h)$ be a flat $2$-torus, identified with the quotient of $\real 2$ by a lattice $L$. The dual lattice $L^{\star}$ is identified precisely with the lattice of harmonic $1$-forms with integral periods. We note that in dimension $2$ one always has that $L^{\star}$ is homothetic to $L$ (up to a congruence in ${\bf O}(2)$): 
$$
L^{\star}\sim \dfrac{1}{\abs{\hat h}}L,
$$
see Subsection \ref{moduli_space}. 
\begin{rem} We recall that $L^{\star}$ is naturally identified with the lattice of integral harmonic $1$-forms (see Lemma \ref{lemma:identification}). Since we are considering the $L^2$-norm on forms, which in this case corresponds to the usual Euclidean norm rescaled by $|\hat h|^{1/2}$ (all harmonic forms on a flat torus have constant length) we deduce that the Jacobian torus of $(\Sigma_1,\hat h)$ is a flat torus of volume $1$ homothetic to $(\Sigma_1,\hat h)$.
\end{rem}

\medskip

It is a classical fact (see for example \cite[\S B.I]{BeGaMa1971}) that every flat torus $\real 2/L$ is homothetic to the one with lattice generated by the vectors:
$$
w_1=(1,0), \quad w_2=(p,q)
$$
with $p\in[0,\frac 12], q\geq p, p^2+q^2\geq 1$. We call $(p,q)$ the {\it parameters} of $\real 2/L$. It follows that the invariants above can all be  expressed in terms of $(p,q)$. In particular:

\begin{thm}\label{intro:flat} Let $(\Sigma_1,h)$ be a $2$-torus, and let $(\Sigma_1,\hat h)$ be the unique flat torus conformal to $h$ (up to homotheties). Assume that $(\Sigma_1,\hat h)$ has parameters $(p,q)$. Then:
$$
\Lambda_1(\Sigma_1,[h])=\dfrac{\pi^2}{q^3}\left((p^2+q^2-p)^2+q^2\right).
$$
Moreover,
$$
\Lambda_1(\Sigma_1)=\dfrac{8\pi^2}{3\sqrt 3},
$$
uniquely attained at the conformal class of the flat $60^0$-rhombic torus 
(which has parameters $(p,q)=(\frac 12,\frac{\sqrt{3}}{2})$).
\end{thm}

Here is the scheme of the paper. In Section \ref{sec:asympt} we prove Theorem \ref{thm:asymp}. Section \ref{sec:tori} will be devoted to the computation of the magnetic spectrum of flat tori. In Section \ref{sec:optimization} we consider some optimization problems for the magnetic ground state energy in genus $1$. In Section  \ref{sec:conf_inv} we prove Theorem \ref{flatisbest} and use it, in combination with the results of Section \ref{sec:optimization}, to prove Theorem \ref{intro:flat}. In Section \ref{bruno} we provide proofs of the statements 1), 2) and 3) at the end of Subsection \ref{sub:spec_inv}. In Section \ref{sec:ntori} we comment on $d$-dimensional flat tori, showing that in this case the ground state spectrum determines the metric uniquely. Finally, section \ref{sec:proof_main} is devoted to the proof  of Theorem \ref{thm_sequences}, which implies in particular 
 Theorem \ref{mainthm}.

\section{Proof of Theorem \ref{thm:asymp}: asymptotics as $A\to 0$}\label{sec:asympt}

Let  $(M,h)$ be a $d$-dimensional, compact Riemannian manifold and let $A$ be a co-closed $1$-form on it. Given $r>0$, we consider the ground state energy $\lambda_1(M,h,rA)$ and compute its asymptotic behavior as $r\to 0$. The proof of Theorem \ref{thm:asymp} is divided in two steps: namely, we establish an upper and a lower bound for $\lambda_1(M,h,rA)$.

\nero For simplicity of notation, as $(M,h)$ is fixed, in what follows we will write
$$
\lambda_k(rA)\doteq \lambda_k(M,h,rA).
$$

\medskip
\textbf{Upper bound.} We simply choose as test function for $\lambda_1(rA)$ in the min-max principle \eqref{minmax} the constant function $u=1$. As $\nabla^{rA}1=-irA^{\sharp}$ (the dual vector field of $A$) we get immediately:
$$
\lambda_1(rA)\le r^2 \frac{\Vert A\Vert^2_{L^2(h)}}{{\rm vol}(M,h)}.
$$
\medskip
\textbf{Lower bound.} In the sequel, for smooth complex valued functions $u,v$, we write
$$
q(u,v)\doteq\int_M \langle\nabla^{rA}u,\overline{\nabla^{rA} v}\rangle_h dv_h\,,\ \ \ q(v)\doteq q(v,v)=\int_M \vert\nabla^{rA} v \vert_h^2 dv_h.
$$

We write the constant function $1=u_1+v$ where $u_1$ is an eigenfunction for $\lambda_1(rA)$ and $v$ is orthogonal to $u_1$ in $L^2(h):=L^2(M,dv_h)$, namely, $\int_M u_1\bar v dv_h=0$. We recall that the second eigenvalue $\lambda_2(rA)$ can be characterized as in \eqref{minmax}, where the minimum is taken among all functions which are $L^2$-orthogonal to the first eigenfunction. Choosing $v$ as test function we have
\begin{equation} \label{ineg:2}
    \frac{q(v)}{\Vert v\Vert_{L^2(h)}^2} \ge \lambda_2(rA).
\end{equation}
By Green' formula:
$$
q(u_1,v)=\int_M \langle\nabla^{rA}u_1,\overline{\nabla^{rA} v}\rangle_h dv_h=\lambda_1(rA)\int_{M}u_1\bar v\,dv_h=0
$$
This implies that, again by Green's formula:
$$
q(v)=q(v,v)=q(1,v)=\int_M\langle \nabla^{rA}1,\overline{\nabla^{rA} v}\rangle_h  dv_h = \int_M \langle \Delta^{rA} 1,\bar v\rangle_h dv_h = r^2 \int_M\vert A\vert_h^2\bar vdv_h, 
$$
note that the last equality follows from \eqref{ml} and the hypothesis $\delta(rA)=r\delta A=0$, so that $\Delta^{rA}1=r^2\abs{A}^2_h$. By the Cauchy-Schwarz inequality:
\begin{equation}\label{CS}
q(v)\le r^2 \left(\int_M \vert A\vert^4_h dv_h\right)^{1/2}\left(\int_M |v|^2dv_h\right)^{1/2}.
\end{equation}

By inequality (\ref{ineg:2}), we have $\norm{v}_{L^2(h)}\leq q(v)^\frac 12(\lambda_2(rA))^{-\frac 12}$ hence, substituting in \eqref{CS}:
\begin{equation}\label{rfour}
q(v)\le \frac{\int_M \vert A\vert_h^4 dv_h}{\lambda_2(rA)}\cdot r^4 
\end{equation}

As $1=u_1+v$ and $q(u_1,v)=0$:

$$
\frac{q(1)}{\Vert 1\Vert_{L^2(h)}^2}=\frac{q(u_1)+q(v)}{\Vert u_1 \Vert_{L^2(h)}^2 + \Vert v\Vert^2_{L^2(h)} } = \frac{\lambda_1(rA)\Vert u_1\Vert_{L^2(h)}^2+q(v)}{\Vert u_1 \Vert_{L^2(h)}^2 + \Vert v\Vert_{L^2(h)}^2 }
$$
which we can rewrite

\begin{align}
\lambda_1(rA)\Vert u_1\Vert_{L^2(h)}^2= \frac{q(1)}{{\rm vol}(M,h)}(\Vert u_1 \Vert_{L^2(h)}^2 + \Vert v\Vert_{L^2(h)}^2)-q(v).
\end{align}

Dividing both sides by $\norm{u_1}_{L^2(h)}^2$
we obtain the following lower bound for $\lambda_1(rA)$:

\begin{align}\label{ineg:3}
    \lambda_1(rA) &= \frac{q(1)}{{\rm vol}(M,h)}+\frac{q(1)}{{\rm vol}(M,h)} \frac{\Vert v\Vert_{L^2(h)}^2}{\Vert u_1\Vert_{L^2(h)}^2}-\frac{q(v)}{\Vert u_1 \Vert_{L^2(h)}^2}\\
    &\ge
    \frac{q(1)}{{\rm vol}(M,h)}- \frac{q(v)}{\Vert u_1 \Vert_{L^2(h)}^2}.
\end{align}

We now observe that, as $r\to 0$, we have $\lambda_2(rA)\to \lambda_2$, the second eigenvalue of the usual Laplace-Beltrami operator, which is positive. Then, 
using (\ref{ineg:2}), \eqref{rfour} and the identity $\norm{u_1}_{L^2(h)}^2={\rm vol}(M,h)-\norm{v}_{L^2(h)}^2$, we see that if $\rho>0$ is sufficiently small there are positive constants $c_1,c_2,c_3,C$ such that, for all $r<\rho$:
$$
q(v)\leq c_1 r^4, \quad \norm{v}_{L^2(h)}^2\leq c_2r^4, \quad \norm{u_1}_{L^2(h)}^2\geq c_3, \quad \dfrac{q(v)}{\norm{u_1}_{L^2(h)}^2}\leq Cr^4.
$$
which implies that

\begin{equation}
\begin{aligned}
\lambda_1(rA)&\geq \frac{q(1)}{{\rm vol}(M,h)}-Cr^4\\
&=\frac{\Vert A\Vert_{L^2(h)}^2}{{\rm vol}(M,h)}r^2-Cr^4
\end{aligned}
\end{equation}

which proves the theorem. \qed

\section{The spectrum of a flat torus}\label{sec:tori}

We study here in detail the spectrum of $\Sigma_1$ with a flat metric. Often we will indicate the compact, orientable, differentiable surface of genus $1$ by $T$ (torus) instead of $\Sigma_1$.

\subsection{Flat torus as a quotient}\label{sub:quotient}

A flat torus is the quotient of $\real 2$ by a lattice $L$ with the inherited Euclidean metric. Precisely, if the lattice is generated by the vectors $w_1=(x_1,y_1), w_2=(x_2,y_2)$, then
$$
L=\{mw_1+n w_2: m,n\in\mathbb Z\}
$$
and
$
(T,\hat h)=(\real 2/L,h_E),
$
where $h_E$ is the Euclidean metric. In this section, we will drop the subscript $h_E$ from the Euclidean scalar product and the Euclidean pointwise norm. The parallelogram generated by $(w_1,w_2)$:
$$
\mathcal P=\{aw_1+bw_2: a,b\in [0,1]\}.
$$
is a {\it fundamental domain} of $(T,\hat g)$. We have then $|\hat h|=|\mathcal P|$, namely, the volume of $(T,\hat h)$ is the (Euclidean) volume of the fundamental domain.

A function $f:T\to \mathbb C$ is obtained from a function $\tilde f:\real 2\to \mathbb C$ which is invariant under the lattice, that is $\tilde f(x+p)=\tilde f(x)$  for all $x\in\mathbb R^2$, $p\in L$. In the next subsection we describe
the set of eigenfunctions of $\lambda_1(T,\hat h,A)$ for every harmonic potential $A$. For that, introduce the {\it dual lattice} $L^{\star}$ of $L$
defined as 
$$
L^{\star}=\{p^{\star}\in\real 2: \scal{p^{\star}}{w}\in\mathbb Z\quad\text{for all}\quad w\in L\}.
$$
If $(w_1,w_2)$ is a basis of $L$,  the pair $(w_1^{\star}, w_2^{\star})$ defined by the conditions: 
$$
\scal{w_i^{\star}}{w_j}=\delta_{ij}
$$
is a basis of $L^{\star}$ called the {\it dual basis of $(w_1,w_2)$}.

\subsection{Eigenfunctions of the magnetic Laplacian} 

The family of eigenfunctions is the same as the one of the Laplace-Beltrami operator for any potential $A$. Given $p^{\star}\in L^{\star}$ consider the function $f:\real 2\to\mathbb C$ defined as
$$
f(x)=e^{2\pi i\scal{x}{p^{\star}}}.
$$
One sees immediately that $f$ is invariant under translations defined by the lattice $L$, hence it defines a function on $T$ (which we still denote by $f$). Writing $x=(x_1,x_2)$ and $p^{\star}=(p_1,p_2)$ we can write
$$
f(x_1,x_2)=e^{2\pi i(p_1x_1+p_2x_2)}
$$
and one sees immediately that 
$$
\Delta f=4\pi^2(p_1^2+p_2^2)f=4\pi^2\abs{p^{\star}}^2f
$$
which shows that $4\pi^2\abs{p^{\star}}^2$ is an eigenvalue of the usual Laplacian. 

\smallskip

Introduce a magnetic potential $A$ on $(T,\hat h)$, which corresponds to a parallel $1$-form on the quotient $T$; its pull-back by the canonical projection $\pi:\real 2\to T$
is then a parallel $1$-form on $\real 2$ which can be written $\pi^{\star}A=2\pi(\alpha dx_1+\beta dx_2)$ for some $(\alpha,\beta)\in\real 2$ and which, by a slight abuse of language, we still denote $A$:
\begin{equation}\label{1form}
A=2\pi(\alpha dx_1+\beta dx_2),
\end{equation}
keeping in mind that, for different lattices $L_1,L_2$, the  differential form in \eqref{1form} represents two different potentials on the respective quotient tori $\real2/L_1, \real 2/L_2$.

We denote by $P_A=(\alpha,\beta)\in\real 2$ and call it the {\it position vector} associated with $A$. We observe:

\begin{lemme}\label{lemma:identification} The potential $A$ on $T=\real 2/L$ has integral flux along every loop in $T$ if and only if $P_A$ belongs to the dual lattice $L^{\star}$.
\end{lemme}

The proof is immediate from the observation that $A$ has integral flux along the loop corresponding to $w\in L$  if and only if $\scal{P_A}{w}\in\mathbb Z$.


\smallskip

We now compute the magnetic spectrum associated to $A$, starting from the definition \eqref{ml}:
 \begin{equation}
 \Delta_Af=\Delta f+\vert A\vert^2f+2i\langle A,df\rangle.
 \end{equation}
From \eqref{1form} we have immediately
$$
\abs{A}^2=4\pi^2(\alpha^2+\beta^2),
$$
and, as $df=2\pi if(p_1dx_1+p_2dx_2)$: 
$$
2i\scal{A}{df}=2i(4\pi^2if)(\alpha p_1+\beta p_2)=-8\pi^2f(\alpha p_1+\beta p_2).
$$
Therefore
$$
\begin{aligned}
\Delta_Af&=4\pi^2 f\Big(p_1^2+p_2^2+\alpha^2+\beta^2-2(\alpha p_1+\beta p_2)\Big)\\
&=4\pi^2 f\abs{P_A-p^{\star}}^2
\end{aligned}
$$
which shows that, for any $p^{\star}\in L^{\star}$ in the dual lattice, we get the magnetic eigenvalue
$$
\lambda(p^{\star})=4\pi^2 \abs{P_A-p^{\star}}^2.
$$
As the family of eigenfunctions thus obtained is a complete system in $L^2$ (it is the same family of eigenfunctions of the usual Laplacian, see e.g., \cite[\S III.B]{BeGaMa1971}), we obtain the following theorem.

\begin{thm}\label{thm:ev_flat} Assume that the flat torus $(T,\hat h)$ is the quotient of $\real 2$ by the lattice $L$. Then a complete set of eigenfunctions of the magnetic Laplacian with potential $A=2\pi(\alpha dx_1+\beta dx_2)$ is indexed by $p^{\star}\in L^{\star}$ as follows:
$$
f_{p^{\star}}(x)=e^{2\pi i\scal{p^{\star}}{x}},
$$
with associated eigenvalue 
$$
\lambda(p^{\star})=4\pi^2\abs{P_A-p^{\star}}^2,
$$
where $P_A=(\alpha,\beta)$. The spectrum is therefore
$$
{\rm spec}(T,\hat h,A)=\{4\pi^2\abs{P_A-p^{\star}}^2: p^{\star}\in L^{\star}\}.
$$
\end{thm}
As a consequence we get the following.

\begin{cor}\label{cor:ev_flat} For the flat torus $(T,\hat h)$ defined by the lattice $L$ the lowest eigenvalue is
$$
\lambda_1(T,\hat h, A)=4\pi^2 \inf_{p^{\star}\in L^{\star}}\abs{P_A-p^{\star}}^2.
$$
\end{cor}

\begin{figure}[H]
\includegraphics[width=0.8\textwidth]{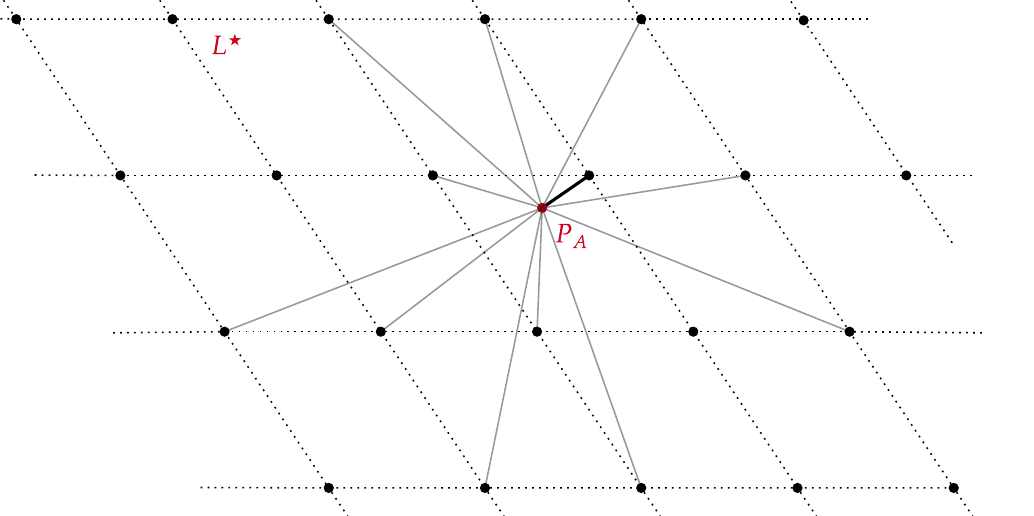}
\caption{The spectrum of a flat torus for a potential $A=2\pi(\alpha\,dx_1+\beta\,dx_2)$ is given by the squared distances  of $P_A=(\alpha,\beta)$ from the points of $L^{\star}$ (i.e., the squared lengths of the segments in the picture), up to a factor $4\pi^2$. The first eigenvalue corresponds to the shortest distance, i.e., to the length of the bold segment in the picture.}
\label{fig:distances}
\end{figure}

\begin{figure}[H]
\includegraphics[width=0.8\textwidth]{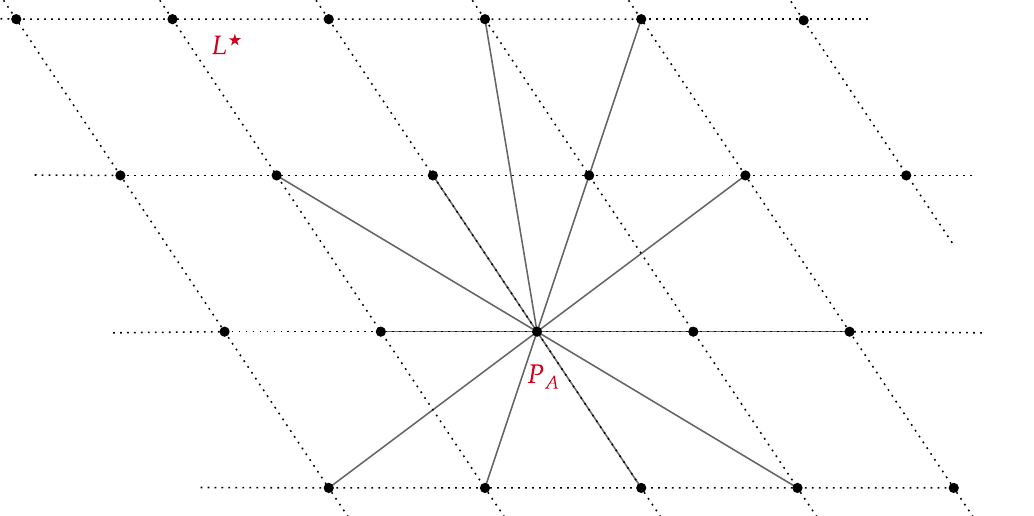}
\caption{When $P_A\in L^{\star}$ we get the usual Laplace-Beltrami spectrum.}
\label{fig:distances2}
\end{figure}

We remark the following facts:
\begin{itemize}
\item The eigenfunctions do not depend on the magnetic potential $A$: they are always the same.
\item The classical Laplacian spectrum is obtained precisely when $P_A\in L^{\star}$ is a point of the dual lattice; in that case $\lambda_1=0$. In other words, the magnetic potential has integer fluxes and hence is gauge equivalent to $A=0$.
\item More generally, the spectrum does not change if we replace $P_A$ with $P_A+p^{\star}$, where $p^{\star}$ is any other point of 
$L^{\star}$. This gives a geometric meaning to the gauge invariance of the magnetic spectrum: if two potentials have fluxes which differ by integers, they give rise to the same magnetic spectrum.
\item  Up to the multiplicative factor $4\pi^2$, the magnetic spectrum associated to the harmonic form $A$  is given by the squares of the distances of $P_A$ to the points of the dual lattice $\mathcal L^{\star}$.  Seen on the quotient flat torus, the magnetic spectrum determines the set of the lengths of all geodesics joining the points $\pi(P_0)$ and $\pi(P_A)$ of $T$, where $\pi$ is the projection onto $T$ and $0$ is the zero-form (this reduces to the marked length spectrum when $\pi(P_A)=\pi(P_0)$, which corresponds to the Laplace-Beltrami operator).
\end{itemize}

\subsection{Moduli space of flat tori}\label{moduli_space}

The following fact  is well-known, see e.g., \cite[\S III.B]{BeGaMa1971}.

\begin{lemme}\label{lem:moduli}
    Up to homotheties, congruences in ${\bf O}(2)$ and change of basis in ${\bf SL}(2,{\mathbb Z})$, the lattice $L$ admits a basis 
$w_1=(1,0), w_2=(p,q)$  with $p\in [0,\frac 12]$, $q\geq p$ and $p^2+q^2\geq 1$. 
\end{lemme}

\begin{figure}[H]
\centering
\includegraphics[width=0.8\textwidth]{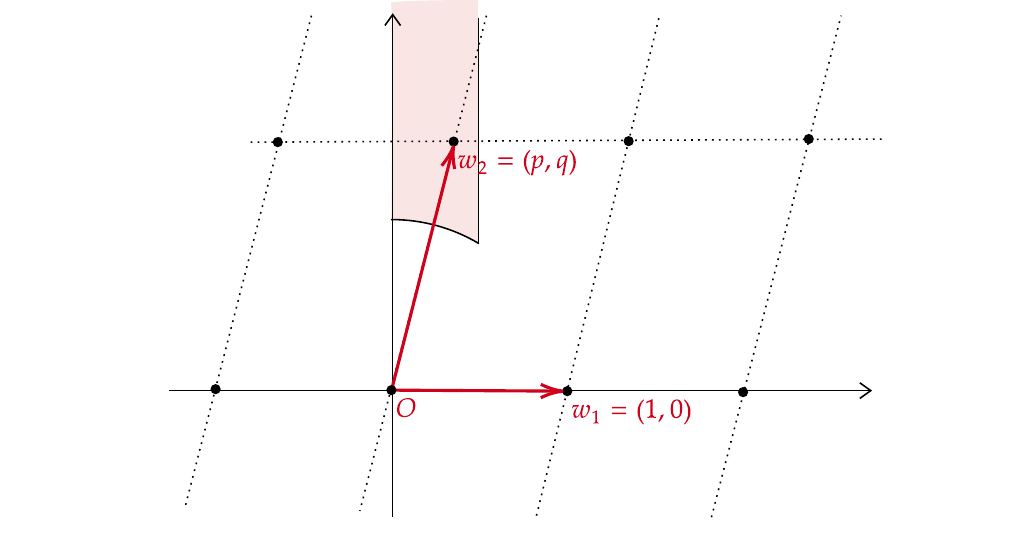}
\caption{Moduli space of flat tori.}
\label{fig:moduli}
\end{figure}

Now, if $L$ has basis $w_1=(1,0), w_2=(p,q)$,  the dual basis $(w_1^{\star},w_2^{\star})$ of $L^{\star}$  defined in Subsection \ref{sub:quotient} is given by 
$$
w_1^{\star}=(1,-\frac{p}{q}), \quad w_2^{\star}=(0,\frac{1}{q}).
$$
However it is convenient to use the basis of $L^{\star}$  given by
\begin{equation}\label{acute_basis}
v_1^{\star}=(1,\frac{1-p}{q}), \quad v_2^{\star}=(0,\frac{1}{q}).
\end{equation}
It is indeed elementary to check that $(v_1^{\star},v_2^{\star})$ is another basis of $L^{\star}$ and that the parallelogram generated by $(w_1,w_2)$ is homothetic to the one generated by $(v_1^{\star},v_2^{\star})$ (up to a congruence in ${\bf O}(2)$): in fact, they have the same angles and proportional sides. Hence, the lattice $L^{\star}$ is homothetic to  $L$ (up to a congruence in ${\bf O}(2)$). Moreover, the area of the parallelogram generated by $(w_1,w_2)$ is $q$, while the area of the parallelogram generated by $(v_1^{\star},v_2^{\star})$ is $\frac{1}{q}$. This shows that $L^{\star}\sim\frac{1}{|\hat h|}L$, where $|\hat h|=|\mathcal P|=q$ is the volume of the fundamental domain of the torus generated by $L$.

\smallskip

Although we will not use it in the paper, it is perhaps interesting to express the spectrum of a flat torus in terms of the fluxes of the $1$-form $A$ with respect to a homology basis (these fluxes in fact determine the $1$-form $A$). In particular, we will take the loops corresponding to the normal basis $(w_1,w_2)$, and we will call the corresponding fluxes $\Phi_1,\Phi_2$. We have
\begin{equation}\label{fluxes}
\twosystem{\Phi_1=\langle(\alpha,\beta),(1,0)\rangle=\alpha}{\Phi_2=\langle(\alpha,\beta),(p,q)\rangle=\alpha p+\beta q}
\end{equation}
We can  write $(\alpha,\beta)$ in term of the fluxes $\Phi_1,\Phi_2$ and of the parameters $p,q$ of the flat torus in the moduli space
$$
\twosystem{\alpha=\Phi_1}{\beta=\langle(\alpha,\beta),(p,q)\rangle=-\frac{p}{q}\Phi_1+\frac{1}{q}\Phi_2}\quad\text{and}\quad \twosystem{\alpha=\Phi_1}{\beta=-\frac{p}{q}\Phi_1+\frac{1}{q}\Phi_2}.
$$
We have the following.
\begin{cor}\label{cor:eig_flux}
$$
|\hat h|\lambda_1(T,\hat h, A)=4\pi^2 \min_{p^{\star}=(p_1,p_2)\in L^{\star}}d_{\hat h}(P_A,L^{\star})^2=4\pi^2 \min_{p^{\star}=(p_1,p_2)\in L^{\star}}\left(q(\Phi_1-p_1)^2+\frac{(p\Phi_1-\Phi_2+qp_2)^2}{q}\right).
$$
\end{cor}

We conclude this section by remarking that one can compute the magnetic spectrum of a $d$-dimensional flat torus $(T^d,\hat h)$: it is $\mathbb R^d/L$ where $L$ is a lattice of full rank. The dual lattice $L^{\star}$ is defined exactly as for $d=2$ and we have the analogous of Theorem \ref{thm:ev_flat}:
\begin{thm}
Assume that the flat torus $(T^d,\hat h)$ is the quotient of $\mathbb R^d$ by the lattice $L$. Then a complete set of eigenfunctions of the magnetic Laplacian with potential $A=2\pi(\sum_{i=1}^d\alpha_idx_i)$ is given by $\{e^{2\pi i\langle p^{\star},x\rangle}:p^{\star}\in L^{\star}\}$, and the spectrum is given by
$$
{\rm spec}(T^d,\hat h,A)=\left\{4\pi^2|P_A-p^{\star}|^2:p^{\star}\in L^{\star}\right\},
$$
where $P_A=(\alpha_1,...,\alpha_d)$.
\end{thm}



\section{Optimization problems for the lowest eigenvalue of flat tori}\label{sec:optimization}

Let $(T,\hat h)$ be a flat $2$-torus. Recall that, if $A=2\pi(\alpha dx+\beta dy)$ and $P_A=(\alpha,\beta)$, then
$$
\lambda_1(T,\hat h,A)=4\pi^2\min_{p^{\star}\in L^{\star}}\abs{P_A-p^{\star}}^2
$$
where $P_A=\dfrac{1}{2\pi}A$. We identify $P_A$ with the point $(\alpha,\beta)\in\real 2$. 
We ask: what is the optimal potential $1$-form $A$, the one which maximizes the normalized lowest eigenvalue? In other words, we want to study the quantity (which is scale-invariant):
$$
\Lambda_1(T,\hat h)\doteq \sup_{A\in {\rm Har}(\hat h)}\abs{\hat h}
\lambda_1(T,\hat h,A)
$$
Clearly, the optimal potential form $A$ corresponds to a point $\hat P_A\in\real 2$ which is at maximum distance to the dual lattice $L^{\star}$, see Figure \ref{fig:optim}. Here $L$ has basis $w_1=(1,0), w_2=(p,q)$. It will be crucial to choose the basis $(v_1^{\star},v_2^{\star})$ defined in \eqref{acute_basis} for the dual lattice $L^{\star}$:
$$
v_1^{\star}=(1,\frac{1-p}{q}), \quad v_2^{\star}=(0,\frac{1}{q}).
$$
With this choice, the triangle $\tau$ generated by $v_1^{\star},v_2^{\star}$ is {\it acute}, namely, all its angles are smaller than $\frac{\pi}{2}$.
We have the following Lemma, the proof of which is postponed at the end of this section.

\begin{defi}
Given an acute triangle $\tau$, its {\rm circumradius} is the radius of the circle (called {\rm circumcircle} passing through the three vertices of $\tau$. The circumradius is also the maximum distance of a point in the triangle to the set of the three vertices. The center of the circumcircle is called {\rm circumcenter}.
\end{defi}

\begin{lemme}\label{lem_triang}
 The point $\hat P_A$ of the triangle $\tau$ generated by $v_1^{\star},v_2^{\star}$ at maximum distance to $L^{\star}$ is  the circumcenter of $\tau$, as in Figure \ref{fig:optim}.
It has coordinates 
$$
\hat P_A=\left(\dfrac{p^2+q^2-p}{2q^2},\frac{1}{2q}\right).
$$
\end{lemme}

Lemma \ref{lem_triang} is no longer true if we use a basis of $L^{\star}$ which generates a triangle that is not acute. 

\smallskip

\begin{figure}
\centering
\includegraphics[width=0.8\textwidth]{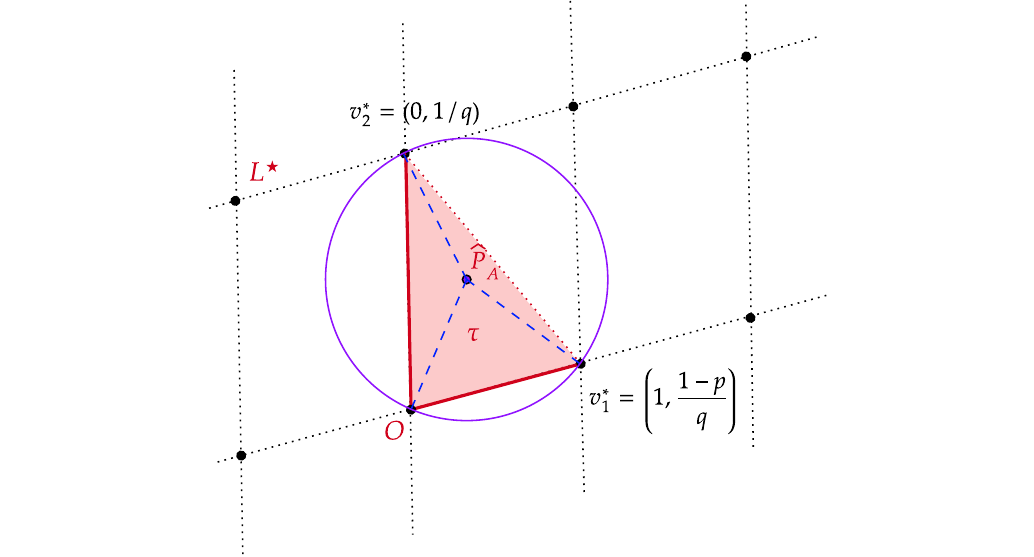}
\caption{The point in the acute triangle $\tau$ at maximum distance from the lattice $L^{\star}$ is the circumcenter of the triangle.}
\label{fig:optim}
\end{figure}

The maximum distance (inradius of $L^{\star}$) satisfies
$$
\mathcal R(L^{\star})^2=\dfrac{(p^2+q^2-p)^2+q^2}{4q^4}.
$$

Note that $\mathcal R(L^{\star})$ is the inradius of $L^{\star}$ for the Euclidean distance. Recall the definition \eqref{inradius_h} of $\mathcal R_{\hat h}(L^{\star})$: it is the maximum $L^2(\hat h)$-distance from the points of $L^{\star}$. Since the $L^2(\hat h)$-distance and the Euclidean distance are homothetic by a factor $|\hat h|^{1/2}$ we have the following relation between $\mathcal R(L^{\star})$ and $\mathcal R_{\hat h}(L^{\star})$:
$$
\mathcal R_{\hat h}^2(L^{\star})=|\hat h|\mathcal R(L^{\star})^2=q\mathcal R(L^{\star})^2.
$$
Therefore we obtain the following

\begin{thm}\label{ftpq} For a flat torus with parameters $(p,q)$ we have:
$$
\Lambda_1(T,\hat h)=\dfrac{\pi^2}{q^3}\left((p^2+q^2-p)^2+q^2\right).
$$
\end{thm}

Finally, we consider:
$$
\Lambda_1(T)=\inf_{\hat h}\Lambda_1(T,\hat h).
$$
Taking the infimum over all flat tori of the above is equivalent to finding the dual lattice of given area having the smallest inradius, that is, 
to find the triangle of given area with smallest circumradius. The problem is equivalent to finding the triangle of largest area inscribed in a fixed circle. This is, evidently, the equilateral triangle, and since the dual lattice is homothetic to the lattice, the minimizing torus corresponds to the parameters $(p,q)=\left(\frac{1}{2},\frac{\sqrt{3}}{2}\right)$. Hence we have:

\begin{thm}\label{ftpq2} One has:
$$
\Lambda_1(T)=\dfrac{8\pi^2}{3\sqrt 3}
$$
attained precisely at the flat $60^0$-rhombic torus (that is, the flat equilateral torus).
\end{thm}

We provide now a proof of Lemma \ref{lem_triang}.




\begin{proof}[Proof of Lemma \ref{lem_triang}]

It is not restrictive to prove  Lemma \ref{lem_triang} for a lattice $L$ with basis $w_1=(1,0)$, $w_2=(p,q)$ with $0\leq p\leq\frac{1}{2}$, $q\geq p$, $p^2+q^2\geq 1$. In fact, the lattice $L^{\star}$ as in Lemma \ref{lem_triang} is homothetic to $L$ and the triangle generated by $v_1^{\star},v_2^{\star}$ is homothetic to the one generated by $w_1,w_2$. Call $\tau$ the (closure of the) triangle with basis $(w_1,w_2)$ (in red in Figure \ref{dim_lem_1}). Note that $\tau$ is {\it acute}: all angles are $\leq\pi/2$. Let $\rho$ be its circumradius and let $\mathcal R(L)$ be the inradius of $L$. We refer to Figure \ref{dim_lem_1} below. Lemma \ref{lem_triang} amounts to proving that
$$
\mathcal R(L)=\rho
$$

{\bf Step 1: $\mathcal R(L)\leq\rho$.} Let  $x\in\mathbb R^2$. It is not restrictive to assume that $x\in\tau$ (up to a translation of the origin and a rotation of $\pi$). We conclude that $\min\{|x|,|x-w_1|,|x-w_2|\}\leq\rho$ which implies that the distance of $x$ from $L$ is at most $\rho$.

\medskip

{\bf Step 2: $\mathcal R(L)=\rho$.} Let $D$ be the {\it open} disk bounded by the circumcircle $C$ of $\tau$. We make the following {\it claim}: there are no lattice points in $D$, while there are at most four lattice points on $C$. Assume that the claim is proved. Taking $x$ to be the circumcenter of $\tau$ we see that $\min\{|x|,|x-w_1|,|x-w_2|\}=\rho$ and that $|x-w|\geq\rho$ for any other $w\in L$. This concludes the proof of Lemma \ref{lem_triang} since we have found $x\in\mathbb R^2$ for which ${\rm min}_{w\in L}|x-w|=\min\{|x|,|x-w_1|,|x-w_2|\}=\rho$.

\medskip

{\bf Proof of the claim.} We prove that no lattice point lies in the open disk $D$ bounded by the circumcircle $C$ of $\tau$, and at most four lattice points belong to $C$.  The circumcenter of $\tau$ lies in its interior, or in the middle point of the hypothenuse (if $\tau$ is right). Let $h$ be the height on the {\it shortest side} of $\tau$, corresponding to $w_1$. We always have 
$$
h\geq\rho.
$$ 
Therefore, only the lattice points on the line $r_0$ through $O$ and $w_1$ and on its parallel $r_1$ through $w_2$ may belong to $D$, see Figure \ref{dim_lem_1} below.

\medskip

Moreover, the only lattice points in $r_0\cap C$ are $O$ and $w_1$ while no lattice point belongs to $r_0\cap D$. On the other hand, $w_2\in r_1\cap C$.  Since $0\leq p\leq\frac{1}{2}$, no lattice point $w_2+kw_1$, $k<0$, belong to $D$. If we prove that $w_2+ w_1$ does not belong $D$ we are done since then no other point $w_2+kw_1$, $k>0$ can belong to $D$. We refer to Figure \ref{dim_lem_1} below.

\medskip  Assume that  $w_2+w_1$ belongs to $D$ and let $v$ its nearest point on $C$. Let $\tau'$ the triangle of vertices $w_1,w_2,w_1+w_2$ which is congruent to $\tau$.  Call $\beta$ the angle of $\tau'$ at $w_1+w_2$, and call $\gamma$ the angle at $v$ of the triangle of vertices $w_1,w_2,v$. Let $\alpha$ be the angle of $\tau$ at $O$. See Figure \ref{dim_lem_1} below.
\begin{figure}[H]
\centering
\includegraphics[width=\textwidth]{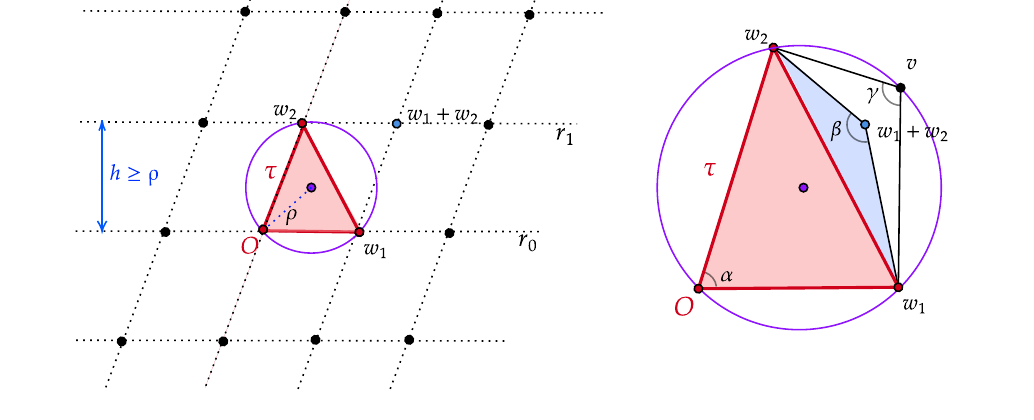}
\caption{}
\label{dim_lem_1}
\end{figure}
We have $\beta>\gamma$ and $\alpha+\gamma=\pi$, hence $\beta>\pi/2$ since $\alpha\leq\pi/2$. This is a contradiction with the fact that $\tau'$ is acute. We conclude that there are no lattice points on $r_1\cap D$. 

\medskip Through the proof we have seen tht $O,w_1,w_2\in C$. By the above argument, if all the angles of $\tau$ are {\it strictly} less that $\pi/2$, we see that $w_1+w_2\not\in C$, therefore $O,w_1,w_2$ are the only lattice points on $C$. If $\alpha=\pi/2$, then $w_1+w_2\in C$ and there are exactly four lattice points in $C$: this last case corresponds to a rectangular torus.

\end{proof}


\section{The conformal invariants in genus one}\label{sec:conf_inv} We now evaluate the conformal invariants introduced in Subsection \ref{sub:spec_inv} when the genus is $1$. Recall their definitions:

\begin{align}
\Lambda_1(\Sigma_g,[h],A)&=\sup_{h'\in [h]}\abs{h'} \lambda_1(\Sigma_g,h',A)\\
\Lambda_1(\Sigma_g,[h])&=\sup_{A:dA=0}\Lambda_1(\Sigma_g,[h],A)\\
\Lambda_1(\Sigma_g)&=\inf_{[h]}\Lambda_1(\Sigma_g,[h])
\end{align}

The following theorem shows that the supremum in the first invariant is attained at the flat metric in the conformal class.

\begin{thm}\label{flatisbest2} Given any metric $h$ on the $2$-torus $\Sigma_1$ one has
$$
\abs{h}\lambda_1(\Sigma_1,h,A)\leq \abs{\hat h}\lambda_1(\Sigma_1,\hat h,A)
$$
where $\hat h$ is the unique (up to homotheties) flat torus in the conformal class of $h$. Equality holds iff $h$ and $\hat h$ are homothetic. In particular:
$$
\Lambda_1(\Sigma_1,[h],A)=\abs{\hat h}\lambda_1(\Sigma_1,\hat h,A).
$$
\end{thm}
\begin{proof} First observe that the metrics $h$ and $\hat h$, being conformal, share the same harmonic $1$-forms. We then take an eigenfunction $u$ associated to $\lambda_1(\Sigma_1,\hat h,A)$ and use it as test-function for $\lambda_1(\Sigma_1,h,A)$.
From the calculation in Section \ref{sec:tori} it follows that all eigenfunctions of flat tori have constant modulus, hence we can take $u$ so that $\abs{u}=1$ on $\Sigma_1$.
We now have:
$$
\begin{aligned}
\abs{h}\lambda_1(\Sigma_1,h,A)
&=\lambda_1(\Sigma_1,h,A)\int_{\Sigma_1}\abs{u}^2dv_h\\
&\leq\int_{\Sigma_1}\abs{\nabla^Au}_h^2dv_h\\
&=\int_{\Sigma_1}\abs{\nabla^Au}_{\hat h}^2dv_{\hat h}\\
&=\lambda_1(\Sigma_1,\hat h,A)\int_{\Sigma_1}\abs{u}^2dv_{\hat h}\\
&=\abs{\hat h}\lambda_1(\Sigma_1,\hat h,A)
\end{aligned}
$$
where in the first inequality we used the min-max principle, and in the second equality the conformal invariance of the magnetic Dirichlet energy (see e.g., \cite[Appendix C]{CPS_AB}). The equality case follows from the expression of the magnetic Laplacian under conformal change. Assume that $h=\phi\hat h$ for $\phi\in C^{\infty}(\Sigma)$, and let $\Delta_{h,A}$ (resp. $\Delta_{\hat h,A}$) be the respective magnetic Laplacians with the same potential form $A$. Then we have
$$
\Delta_{h,A}=\dfrac{1}{\phi}\Delta_{\hat h,A}.
$$
If equality holds, then $u$ (which is an eigenfunction of $\Delta_{\hat h,A}$) must be an eigenfunction also of $\Delta_{h,A}$ and this forces $\phi$ to be constant, hence $h,\hat h$ are homothetic. Conversely, if $h,\hat h$ are homothetic it is immediate that the respective normalized eigenvalues are the same. 

\end{proof}

Theorem \ref{intro:flat} is now clear. In fact, it is just a consequences of Theorems \ref{ftpq}, \ref{ftpq2} and \ref{flatisbest2}.

\section{Some remarks on the conformal invariants}\label{bruno}

\subsection{The infimum in the conformal class is zero}\label{sub:inf_0}

\begin{thm}\label{thm_inf_0} We have, for a given conformal class of metrics $[h]$ on $\Sigma_g$ and a given closed potential $A$:
$$
\inf_{h'\in [h]}\abs{h'}\lambda_1(\Sigma_g,h',A)=0.
$$
\end{thm}

\begin{proof} The proof amounts to construct a family of metrics $h_{\epsilon}$, conformal to $h$, such that the normalized ground state energy $\abs{h_{\epsilon}}\lambda_1(\Sigma_g,h_{\epsilon},A)\to 0$ as $\epsilon\to 0$.
Pick a point $p \in \Sigma_g$ and consider the ball $B(p,\epsilon)$ of radius $\epsilon$. When $\epsilon \to 0$, the ball $B(p,\epsilon)$ is quasi-isometric to a Euclidean ball of radius $\epsilon$ (see  \cite[Lemma 2.3]{CoSo2003}), and one can show that, for our purposes, we can assume that being only quasi-isometric to a Euclidean ball does not affect the final statement (see \cite[p. 343]{CoSo2003}).

Thus, we can assume that each such ball is Euclidean. By stereographic projection, $B(p,\epsilon)$ is mapped conformally to the domain $\Omega_{\epsilon}\doteq \mathbb S^2 \setminus B(\epsilon)$, the complement of a ball of radius $\epsilon$ on the round sphere of radius $1$. We then replace $B(p,\epsilon)$ with $\Omega_{\epsilon}$ to obtain a metric $h_{\epsilon}$ in the conformal class of $h$ (see this construction in \cite[Par. 3]{CoSo2003}). As $\Omega_{\epsilon}$ is simply connected, the potential $A$ is exact there, and up to a gauge transformation, the restriction of $A$ to $\Omega_{\epsilon}$ may then be supposed to be equal to $0$.  Take $u$ to be the first Dirichlet eigenfunction of $\Omega_{\epsilon}$; since it has support in $\Omega_{\epsilon}$, we see that $\Delta_A=\Delta u$. It is well known that the first eigenvalue of the Dirichlet problem on $\Omega_{\epsilon}$, say $\lambda(\epsilon)$, tends to $0$ with $\epsilon$. We now extend $u$ by zero outside $\Omega_{\epsilon}$  and use this extension as test-function for $\lambda_1(\Sigma_g,h_{\epsilon},A)$; we get
$$
\lambda_1(\Sigma_g,h_{\epsilon},A)\leq\lambda(\epsilon)
$$
which tends to zero with $\epsilon$. Note that, in the conformal change, the volume of $h_{\epsilon}$ is controlled from above: $\abs{h_{\epsilon}}\leq c\abs{h}$, hence the normalized eigenvalue tends to zero as well, which shows the assertion.
\end{proof}

\subsection{A uniform bound for hyperbolic metrics}

\begin{thm}\label{firstex} There is a constant $C=C(g)<+\infty$ depending only on the genus $g$, such that, for any hyperbolic metric $h$ on $\Sigma_g$ and any harmonic potential $A$, one has:
$$
\abs{h}\lambda_1(\Sigma_g,h,A)\leq C(g).
$$
\end{thm}

\begin{proof}

Every hyperbolic surface has a point $p$ with injectivity radius $\rho >0$, where $\rho$ is the Margulis constant (see \cite{BaGrSc85}). 
Then, $B(p,\rho)$ is an embedded hyperbolic disc of radius $\rho$; since it is simply connected, $A$ is exact on $B(p,\rho)$ so that,  up to a gauge transformation, one can assume that $A$ is $0$ on $B(p,\rho)$. 
As in Example 1, a universal upper bound of $\lambda_1(\Sigma_g,h,A)$ is then given by the first Dirichlet eigenvalue of $B(p,\rho)$, denoted $\lambda(\rho)$ and not depending on $p$. As the area of a hyperbolic surface of genus $g$ is $4\pi(g-1)$, there is a universal upper bound 
$$
\vert h\vert\lambda_1(\Sigma_g,h,A)\leq C(g)\doteq 4\pi(g-1)\lambda(\rho)
$$ 
as well. We stress that this upper bound is the same for all compact surfaces with curvature $-1$ and does not depend on the conformal class.

\end{proof} 

\subsection{The normalized eigenvalue can be very large}

\begin{thm}\label{secondex} There exist a closed $1$-form $A$ on $\Sigma_g$ and   a sequence $h_n$ of metrics on $\Sigma_g$ such that 
$$
\lim_{n\to\infty} \abs{h_n}\lambda_1(\Sigma_g,h_n,A)=+\infty.
$$
\end{thm}

Let us begin the proof by fixing a hyperbolic surface $(\Sigma_g,h)$ that we write as a union of pants having, as boundary, closed geodesics of length $1$ (see Figure \ref{fig:Example 1} \ and \cite{Bu92} for this type of construction). We fix a canonical basis $\chi_1,...,\chi_{2g}$ of the homology of $\Sigma_g$ (see Subsection \ref{sub:can_bas} for the definition), and as potential form we choose  the closed form $A$ with flux $1/2$ along $\chi_1$ and $\chi_2$ and flux $0$ along $\chi_j$, $j\ge 3$.

\begin{figure}[H]
\centering  \includegraphics[width=\textwidth]{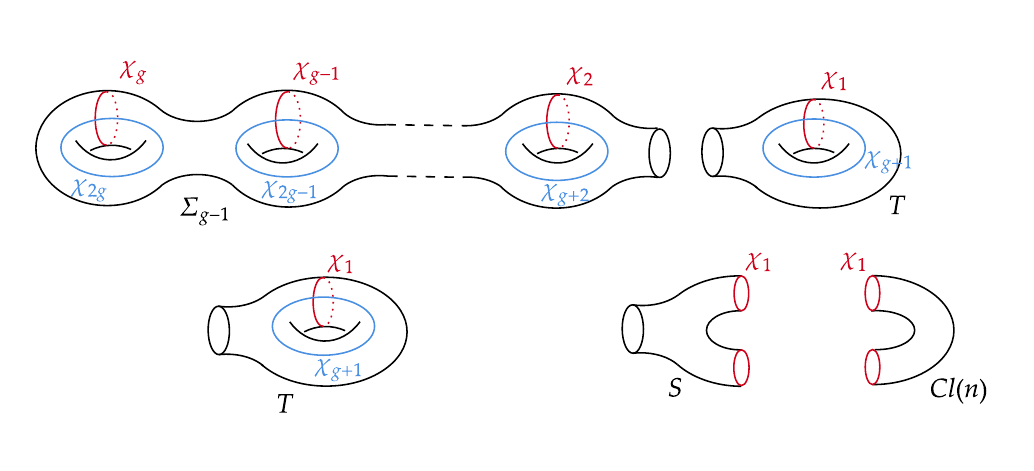}
 \caption{The surface $\Sigma_g$ and its splitting}
  \label{fig:Example 1}
  \end{figure}


\smallskip

We split $\Sigma_g$ as the union of: a hyperbolic surface of genus $g-1$ with one geodesic boundary piece of length $1$, which we denote by $\Sigma_{g-1}$, and a hyperbolic torus $T$ with geodesic boundary of length $1$.


\smallskip

We first work on the surface $\Sigma_{g-1}$. As the restriction of $A$ has flux $1/2$ around $\chi_2\subset \Sigma_{g-1}$, the first eigenvalue of the magnetic Laplacian with Neumann boundary condition, denoted by $\nu_1(\Sigma_{g-1},A)$ is fixed and strictly positive: in fact it is known, by the Shigekawa argument we referred to in the introduction, that the first Neumann eigenvalue is $0$ if and only if all the fluxes are integers (see \cite[(ii) of Proposition 3.1]{shi87} and \cite[(1.7) of Theorem 1.7]{HHHO1999}).

\medskip

Next, we cut the torus $T$ along two closed geodesics homotopic to $\chi_1$ as in Figure \ref{fig:Example 1}, in order to obtain a sphere with three geodesic boundary pieces of length $1$ denoted by $S$. Again, this part is fixed, has a non-integral flux and, consequently, its first Neumann eigenvalue $\nu_1(S)$ is strictly positive.

\medskip
Between the two geodesic boundaries of $S$ corresponding to $\chi_1$, one can glue a flat cylinder $Cl(n)$ of length $n$ and constant cross-section length $1$. One can perform a suitable smoothing on $S$ so that the resulting metric $h_n$ on the union of the three pieces is smooth. Thus, we arrive at a splitting:

\begin{equation}\label{splitting}
\Sigma_g=\Sigma_{g-1}\cup S\cup Cl(n).
\end{equation}

The first Neumann eigenvalue $\nu_1(Cl(n))$ is positive and controlled from below by a positive constant $C$, not depending on  $n$ by \cite[Theorem 1]{CoSa2018},  (note that the variable $L$ in Theorem 1 takes the value $1$ in our situation). 
It is an easily verified fact that the Rayleigh quotient of a function $u$ on a manifold which is the union of pieces $N_1,\dots, N_k$ is bounded below by the minimum of the eigenvalues $\nu_1(N_1), \dots, \nu_1(N_k)$. Given the splitting \eqref{splitting}, the Rayleigh quotient  of a function $u$ on $\Sigma_g$ is bounded from below by the minimum of $\nu_1(\Sigma_{g-1})$, $\nu_1(S)$, $\nu_1(Cl(n))$. We have just seen that these eigenvalues have a uniform positive lower bound independent of $n$; as the area of $h_n$ goes to $\infty$ with $n$, the first normalized eigenvalue diverges to $\infty$ as well.

\section{Higher dimensional flat tori}\label{sec:ntori}
A $d$-dimensional flat Riemannian torus is the quotient of $\mathbb R^d$ by a lattice $L$ with the inherited Euclidean metric. Precisely, if the lattice is generated by $d$ linearly independent vectors $w_1,...,w_d$, then
$$
L=\left\{\sum_{i=1}^nk_iw_i\,:\,k_i\in\mathbb Z\right\}.
$$
 We also write $(T^d,\hat h)=(\mathbb R^d/ L,h_E)$, where $h_E$ is the Euclidean metric.  The spectrum of $d$-dimensional flat tori is computed as in Section \ref{sec:tori} but we recall it here for completeness:

\begin{thm}
Assume that the flat torus $(T^d,\hat h)$ is the quotient of $\mathbb R^d$ by the lattice $L$. Then a complete set of eigenfunctions of the magnetic Laplacian with potential $A=2\pi(\sum_{i=1}^d\alpha_idx_i)$ is given by $\{e^{2\pi i\langle p^{\star},x\rangle}:p^{\star}\in L^{\star}\}$, and the spectrum is given by
$$
{\rm spec}(T^d,\hat h,A)=\left\{4\pi^2|P_A-p^{\star}|^2:p^{\star}\in L^{\star}\right\},
$$
where $P_A=(\alpha_1,...,\alpha_d)$.
The lowest eigenvalue is
$$
\lambda_1(T^d,\hat h, A)=4\pi^2 \inf_{p^{\star}\in L^{\star}}\abs{P_A-p^{\star}}^2.
$$
\end{thm}

 We prove that for flat $d$-dimensional tori, ground state-isospectrality implies isometry:
 \begin{thm}\label{isosp_ntori} Assume that the flat metrics $h,h_0$ on $T^d$ have the same ground state spectrum.  Then $h_0$ and $h$ are isometric.
\end{thm}

\begin{proof}

Let $(T^d,h)$ be a flat torus. Its Jacobian torus is ${\rm Jac}(T^d,h)={\rm Har}(h)/L^{\star}$, where now $L^{\star}$ is generated by a dual basis $(\alpha_1,...,\alpha_d)$ of a given homology basis. We can identify the Jacobian torus of $(T^d,h)$ with $\real d/\mathbb Z^d$ with the metric
$$
\Gamma_{jk}(h)=\scal{\alpha_j}{\alpha_k}_{L^2(h)},
$$
where we take $\alpha_j=dx_j$, $j=1,...,d$.
In other words, the metric is $h=\sum_{j,k}\Gamma_{jk}dx_jdx_k$. The matrix $\Gamma(h)$ with entries $\Gamma_{jk}$ is the {\it Gram matrix} associated to the metric $h$ (see also Subsection \ref{sub:can_bas}).

\smallskip

Let $h_0,h$ be two flat metrics on $T^d$. It is clear that if $\Gamma(h)=c\Gamma(h_0)$ then the respective Jacobian tori are homothetic. Since ${\rm Jac}(T^d,h)$ is homothetic to the dual torus of $(T^d,h)$, taking again duals, we conclude that $h$ and $h_0$ are homothetic: $h=bh_0$. Now observe that $\lambda_1$ rescales as in the Laplace-Beltrami case, that is:
$$
\lambda_1(T^d,h,\frac 1n(\alpha_j+\alpha_k))=\frac{1}{b^2}\lambda_1(T^d,h_0,\frac 1n(\alpha_j+\alpha_k))
$$ 
for all $j,k,n$.
Hence, by ground state-isospectrality we conclude that $b=1$, i.e., $h$ and $h_0$ are isometric.

\smallskip It remains to show that if $h_0$ and $h$ have the same ground state spectrum, then
$$
\Gamma(h)=c\Gamma(h_0).
$$
From the consequence \eqref{eltwonorm} of Theorem \ref{thm:asymp} we have the asymptotics
$$
\lim_{r\to 0}\dfrac{\lambda_1(T^d,h,r(\alpha_j+\alpha_k))}{r^2}=
\dfrac{1}{\abs h}\norm{\alpha_j+\alpha_k}^2_{L^2(h)}
=\frac{
\Gamma_{jj}(h)+\Gamma_{kk}(h)+2\Gamma_{jk}(h)}{\abs{h}}.
$$
We have the analogous limits for $h_0$. Taking $r=\frac 1n$ and passing to the limit as $n\to\infty$ we have, by our hypothesis:
$$
\abs{h_0}\left(\Gamma_{jj}(h)+\Gamma_{kk}(h)+2\Gamma_{jk}(h)\right)=\abs{h}\left(\Gamma_{jj}(h_0)+\Gamma_{kk}(h_0)+2\Gamma_{jk}(h_0)\right)
$$
for all $j,k$. It is an easy exercise to conclude that
$$
\Gamma_{jk}(h)=\dfrac{\abs{h}}{\abs{h_0}}\Gamma_{jk}(h_0)
$$
for all $j,k$. This concludes the proof.

\end{proof}

 We can also consider the spectral invariants defined in Subsection \ref{sub:spec_inv}.
 Most of the proofs of Sections \ref{sec:optimization}  and \ref{sec:conf_inv} apply without major modifications. In particular we have an analogous of Theorem \ref{flatisbest2}:
 \begin{thm}
 Let $\hat h$ be a flat metric on $T^d$ and let $h\in[\hat h]$. Then, for any closed $1$-form $A$ on $T^d$ we have
\begin{equation}\label{up_bd_n}
|h|\lambda_1(T^d,h,A)\leq|\hat h|\lambda_1(T^d,\hat h,A),
 \end{equation}
 \end{thm}
 Through this section we  adopt the notation $|h|$ to denote the ($d$-dimensional) volume of $(T^d,h)$, namely $|h|:={{\rm vol}(T^d,h)}$.
 Note that, contrarily to the $2$-dimensional case, we need to restrict to conformally flat metrics.  If we restrict to conformally flat metrics, it makes sense to define
$$
\Lambda_1(T^d,[h],A)=\sup_{h_0\in[h]}|h_0|(T^d,h_0,A),
$$
$$
\Lambda_1(T^d,[h])=\sup_{A:dA=0}\Lambda_1(T^d,[h],A)
$$
and
$$
\Lambda_1(T^d)=\inf_{[h]}\Lambda_1(T^d,[h]).
$$
Then we have that if $h$ is conformally flat:
$$
\Lambda_1(T^d,[h])=4\pi^2|\hat h|^{2/d}\mathcal R( L^{\star})^2,
$$
where $\mathcal R( L^{\star})$ is the inradius of the dual lattice $ L^{\star}$.

\medskip

On the other hand, the invariant $\Lambda_1(T^d)$ is a rather complicate object, already in three dimensions. This problem is related to the so-called {\it Delaunay triangulation} (see e.g., \cite{raj_94}) of a point set.
Delaunay triangulations have been extensively studied in $\mathbb R^2$ and also in higher dimensions, and are a central topic in Computational Geometry as they enjoy several optimality properties.  The interested reader may refer to \cite{raj_94} and references therein.


\section{Proof of Theorem \ref{thm_sequences}}\label{sec:proof_main} We assume through this section that $\Sigma_g$ is the (unique) differentiable compact, orientable surface of genus $g$. The pair $(\Sigma_g,h)$ denotes the  surface $\Sigma_g$ endowed with the Riemannian metric $h$. For simplicity of notation, as the genus is fixed in the discussion, we simply write
$$
\lambda_1(h,A)\doteq \lambda_1(\Sigma_g,h,A).
$$
We prove the main theorem for an arbitrary genus $g$. When $g=1$ we decided to provide, at the end of the paper, a simple, self-contained proof using only the uniformization theorem; this is done to illustrate the machinery needed in the general case in the simplest possible geometric situation.  We mainly refer to the classical textbook by Farkas and Kra \cite{FK}.

\subsection{Canonical homology basis and Gram matrix}\label{sub:can_bas} We fix a so-called {\it canonical homology basis} $\chi=(\chi_1,\dots,\chi_{2g})$ in $H_1(\Sigma,\mathbb Z)$. This means that $\chi$ is made up of two families of closed curves $a_j,b_j$ so that
$$
(\chi_1,\dots,\chi_{2g})\doteq(a_1,\dots,a_g,b_1\dots,b_g)
$$
and, if the dot denotes the oriented intersection number (the sum of the number of intersections, counted $+1$ if the orientation of the intersection agrees with the orientation of $\Sigma$, and $-1$ otherwise), we have the relations:
$$
a_j\cdot a_k=0, \quad b_j\cdot b_k=0, \quad a_j\cdot b_k=\delta_{jk}, \quad b_j\cdot a_k=-\delta_{jk}.
$$
It follows that the $2g\times 2g$ intersection matrix $J_{jk}=\chi_j\cdot \chi_k$ of a canonical homology basis is
$$
J=\twomatrix 0{I_g}{-I_g}0
$$
Associated to $\chi$ there is a dual basis of the integral cohomology space $H^1(\Sigma,\mathbb Z)$ which is made-up of closed $1$-forms 
$$
\alpha=(\alpha_1,\dots,\alpha_{2g})
$$
so that, by definition, $\int_{\chi_j}\alpha_k=\delta_{jk}$. The dual basis is determined up to an exact $1$-form on each component. The cohomology class of $\alpha$ is then a basis for the cohomology space 
$H^1(\Sigma_g,\mathbb Z)$ and also of $H^1(\Sigma_g,\mathbb R)$ (over $\mathbb R$). Of course, by the Hodge theorem, we can assume that $\alpha$ consists of harmonic forms. 

\smallskip

Recall from Subsection \ref{sub:spec_inv} that the dual basis $\alpha$ generates a lattice in ${\rm Har}(h)$ denoted $L^{\star}$. The quotient
$$
{\rm Jac}(h)\doteq {\rm Har}(h)/L^{\star}
$$
endowed with the $L^2(h)$ metric, is a flat $2g$-dimensional torus, called {\it Jacobian torus} of $h$. 

\nero The Jacobian torus depends only on the conformal class of $h$. 

\smallskip 

It is clear that ${\rm Jac}(h)$ is isometric with $\real{2g}/\mathbb Z^{2g}$ endowed with the metric $\Gamma_{jk}dx_idx_j$ where 
$$
\Gamma_{jk}=\scal{\alpha_j}{\alpha_k}_{L^2(h)}.
$$
The matrix $\Gamma=\Gamma(h)$ with entries $\Gamma_{jk}$ is called the  {\it Gram matrix} associated to the metric $h$; it is uniquely determined by the canonical homology basis we started with and the conformal class of $h$.

\medskip

 The proof of Theorem \ref{mainthm} is divided in two steps:

\begin{itemize}
\item  {\bf Step 1.} We first prove that if two metrics $h,h_0$ have the same ground state spectrum, then they have the same volume and the same Gram matrix: $\Gamma(h)=\Gamma(h_0)$.  
In particular, ${\rm Jac}(h)={\rm Jac}(h_0)$.

The proof uses the asymptotics of $\lambda_1(h,rA)$ as $r\to 0$, proved in Section \ref{sec:asympt},  and the fact that the fixed homology basis is canonical, so that the associated Gram matrices have unit determinant (thanks to the so-called Riemann period relations). 
\item {\bf Step 2.} We prove that if two metrics $h$ and $h_0$ satisfy ${\rm Jac}(h)={\rm Jac}(h_0)$ then they are conformal.

Here we use the conformal structure of $(\Sigma,h)$ and a famous result in Riemann surface theory, the theorem of Torelli. 
\end{itemize}

\subsection{Proof of Step 1}\label{step1}
We prove here the following theorem:
\begin{thm} \label{firstmain} Assume that the metrics $h$ and $h_0$ have the same ground state spectrum. 
Then $h$ and $h_0$ have the same volume and $\Gamma(h)=\Gamma(h_0)$.
In particular, ${\rm Jac}(h)={\rm Jac}(h_0)$.
\end{thm} 
This is done in three steps: first we prove that the Gram matrices are homothetic; then we study in more detail the structure of the Gram matrix; finally we use this information to conclude.

\subsubsection{The Gram matrices are homothetic} 

\begin{thm}\label{homothety}  Assume $h$ and $h_0$ have the same ground state spectrum, and let $\Gamma(h)$ (resp. $\Gamma(h_0)$) be the respective Gram matrices relative to the same canonical homology basis $\chi$. Then
$$
\Gamma(h)=c\Gamma(h_0), \quad\text{with}\quad c=\frac{\abs h}{\abs{h_0}}.
$$
\end{thm}
\begin{proof} Given the basis $\alpha=(\alpha_1,\dots,\alpha_{2g})$ of $H^1(\Sigma_g,\mathbb R)$ introduced above, we project this basis orthogonally on ${\rm Har}(h)$ and obtain the basis
$$
\alpha^h=(\alpha^h_1,\dots,\alpha^h_{2g})
$$
of ${\rm Har}(h)$. Since $\alpha^h_k$ differ from $\alpha_k$ by an exact form, we see:
$$
\oint_{\chi_j}\alpha^h_k=\oint_{\chi_j}\alpha_k=\delta_{jk}
$$
that is, $\alpha^h$ is the unique basis of ${\rm Har}(h)$ which is dual to $\chi$.
In a similar manner we obtain the basis 
$$
\alpha^{h_0}=(\alpha^{h_0}_1,\dots,\alpha^{h_0}_{2g})
$$
which is the unique basis of ${\rm Har}(h_0)$ dual to $\chi$. By gauge invariance and our hypothesis we have, for all $j,k,n$:
\begin{equation}\label{gramone}
\begin{aligned}
\lambda_1(h,\frac 1n(\alpha^h_j+\alpha^h_k))&=
\lambda_1(h,\frac 1n(\alpha_j+\alpha_k))\\
&=\lambda_1(h_0,\frac 1n(\alpha_j+\alpha_k))\\
&=\lambda_1(h_0,\frac 1n(\alpha^{h_0}_j+\alpha^{h_0}_k))
\end{aligned}
\end{equation}

We now apply Theorem \ref{thm:asymp} to the $h$-harmonic form $\frac1n(\alpha^h_j+\alpha^h_k))$ and get:
$$
\begin{aligned}
\lim_{n\to\infty}n^2\lambda_1(h,\frac1n(\alpha^h_j+\alpha^h_k))&=\dfrac{1}{\abs{h}}\norm{\alpha^h_j+\alpha^h_k}_{L^2(h)}^2\\
&=\dfrac{1}{\abs{h}}\Big(\Gamma_{jj}(h)+\Gamma_{kk}(h)+2\Gamma_{jk}(h)\Big).
\end{aligned}
$$
We apply Theorem \ref{thm:asymp} to the $h_0$-harmonic form $\frac1n(\alpha^{h_0}_j+\alpha^{h_0}_k))$ and get the corresponding equality; taking into account the equality in \eqref{gramone}, we have, for all $j,k$
$$
\dfrac{1}{\abs{h}}\Big(\Gamma_{jj}(h)+\Gamma_{kk}(h)+2\Gamma_{jk}(h)\Big)=\dfrac{1}{\abs{h_0}}\Big(\Gamma_{jj}(h_0)+\Gamma_{kk}(h_0)+2\Gamma_{jk}(h_0)\Big)
$$
and it is an easy exercise to conclude that
$$
\Gamma_{jk}(h)=\dfrac{\abs{h}}{\abs{h_0}}\Gamma_{jk}(h_0)
$$
for all $j,k$.
\end{proof}

\subsubsection{Riemann relations, the Hodge-star operator and the structure of the Gram matrix}

Here we fix a generic metric $h$ (and the canonical homology basis $\chi$) and write simply $\Gamma$ for $\Gamma(h)$. In what follows, we omit the subscript $h$: therefore, it is tacitly assumed that the operator $\star$ refers to $\star_h$, the Hodge-star operator for the metric $h$. Similarly, we denote $(\alpha_1,\dots,\alpha_{2g})$ the basis of ${\rm Har}(h)$ dual to $\chi$.  From \cite[p. 58]{FK} we see:
$$
J_{kj}=-\scal{\alpha_k}{\star\alpha_j}.
$$
The operator $\star$ preserves harmonic forms. Let $G$ be its matrix in the basis $(\alpha_1,\dots,\alpha_{2g})$. Writing the basis as a column vector (as in \cite{FK}), we see:
$$
\Tre {\star\alpha_1}{\vdots}{\star\alpha_{2g}}=G\Tre {\alpha_1}{\vdots}{\alpha_{2g}}
$$
which means
$
\star\alpha_k=\sum_{m=1}^{2g}G_{km}\alpha_m.
$
Now:
$$
\begin{aligned}
\Gamma_{kj}&=\scal{\alpha_k}{\alpha_j}\\
&=\scal{\star\alpha_k}{\star\alpha_j}\\
&=\sum_{m=1}^{2g}G_{km}\scal{\alpha_m}{\star\alpha_j}\\
&=-\sum_{m=1}^{2g}G_{km}J_{mj}
\end{aligned}
$$
that is
$
\Gamma=-GJ
$
and since $J^2=-I$ we conclude:

\begin{lemme} The matrices $G$ and $\Gamma$ are related by $G=\Gamma J$.
\end{lemme}

Now we write $\Gamma$ as a block matrix
$
\Gamma=\twomatrix ABCD.
$
Since $\Gamma$ represents a Riemannian metric, $\Gamma$ is symmetric and positive; therefore we have
$$
A=A^t, \quad D=D^t, \quad C=B^t,
$$
moreover $A>0$ and $D>0$. Hence we can write
\begin{equation}\label{gram}
\Gamma=\twomatrix AB{B^t}D
\end{equation}
which implies that
\begin{equation}\label{Gmatrix}
G=\Gamma J=\twomatrix AB{B^t}D\twomatrix 0I{-I}0=\twomatrix {-B}A{-D}{B^t}.
\end{equation}
Knowing that $G^2=-I$ we obtain the following identities:
$$
AD-B^2=I, \quad DB=B^tD, \quad BA=AB^t.
$$
We summarize as follows.

\begin{thm} Let $\Gamma=\twomatrix AB{B^t}D$ be the Gram matrix in a given metric $h$. Then:
\begin{enumerate}[i)]
\item The matrix of the star operator $\star_h$ is $G=\twomatrix {-B}A{-D}{B^t}$;

\item We have the identities
$$
AD-B^2=I, \quad DB=B^tD, \quad BA=AB^t.
$$
\end{enumerate}
\end{thm}

\smallskip

\begin{rem}
Recall the definition of real symplectic matrices:
$$
{\rm Sp}_{2g}(\mathbb R)=\{P\in M_{2g\times 2g}(\reals): P^tJP=J\}.
$$
The Riemann period relations $ii)$ above are equivalent to the identity
$
\Gamma J\Gamma=J
$
and since $\Gamma=\Gamma^t$ and $\Gamma$ is positive we have:

\begin{thm}\label{detone} The Gram matrix $\Gamma$ is a (real) symplectic matrix of determinant $1$. In particular, the Jacobian torus of any Riemannian surface $(\Sigma,h)$ is a flat $2g$-dimensional torus of unit volume.  
\end{thm}

To finish the remark, recall Siegel's space (see \cite{BuSa}):
$$
\mathcal H_{2g}=\{P\in \mathcal P_{2g}:PJP=J\}
$$
where $ \mathcal P_{2g}$ (positive square matrices of determinant $1$) is acted upon by ${\rm Sp}_{2g}(\mathbb Z)$. The quotient
$
{\rm Sp}_{2g}(\mathbb Z)/\mathcal H_{2g}
$
parametrizes the principally polarized abelian varieties of complex dimension $g$. Note that the Gram matrix $\Gamma\in \mathcal H_{2g}$.
\end{rem}
\subsubsection{Conclusion of the proof of Theorem \ref{firstmain}}
From Theorem \ref{homothety} we know that $\Gamma(h)=\frac{\abs{h}}{\abs{h_0}}\Gamma(h_0)$. Since $\Gamma(h)$ and $\Gamma(h_0)$ have both determinant $1$ by Theorem \ref{detone} we must have $\abs h=\abs{h_0}$ and then $\Gamma(h)=\Gamma(h_0)$.
\qed


\subsection{Proof of Step 2: conformal structure and Torelli's theorem} To finish the proof of the main result, we will prove the following:

\begin{thm}\label{secondmain} Assume that the metrics $h$ and $h_0$ satisfy ${\rm Jac}(h)={\rm Jac}(h_0)$. Then $h$ and $h_0$ are conformal. 
\end{thm}

In fact, it is enough to show that $h$ and $h_0$ have the same Gram matrix with respect to the same fixed canonical homology basis $\chi$. The proof of Theorem \ref{secondmain} is divided into three steps. The main tool that we use is a fundamental result in Riemann surfaces theory, Torelli's theorem, which relates the conformal structure of a Riemann surface with its normalized period matrix. Hence we will start by associating in a natural way a Riemann surface to a metric $h$; then we will describe the normalized period matrix of the Riemann surface with respect to a canonical homology basis; finally we prove that the Gram matrix determines the normalized period matrix and conclude.

\subsubsection{A Riemann surface naturally associated to the metric $h$ }

This subsection is quite standard, we add it for convenience of the reader. Again let $(\Sigma,h)$ be an oriented Riemannian manifold of dimension $2$. There is a natural way to introduce an almost complex structure $J_h$ on $\Sigma$, by setting
$$
J_h\doteq\star_h,
$$
where $\star_h$ is the Hodge-star operator induced by the Riemannian metric, which is globally defined because $\Sigma$ is orientable. 

\begin{prop} The almost complex structure defined above is integrable, and uniquely defines a Riemann surface, denoted $(\Sigma,J_h)$. 
\end{prop}

Needless to say, the Hodge-star operator and the family of harmonic $1$-forms on $(\Sigma,h)$ and its associated Riemann surface $(\Sigma,J_h)$ coincide. 

\begin{proof} According to the  Newlander-Nirenberg theorem, $J$ is integrable if and only if the Nijenhuis tensor
$$
N(X,Y)=[JX,JY]-J[JX,Y]-J[X,JY]-[X,Y]
$$
acting on the vector fields $X,Y$ vanishes identically.  Introduce local isothermal coordinates in a neighborhood of any given point, so that the metric is expressed as $ds^2=\lambda(x,y)(dx^2+dy^2)$. As the Hodge-star operator is conformally invariant, we see that
$$
J\derive{}{x}=\derive{}{y}, \quad J\derive{}{y}=-\derive{}{x},
$$
and $N$ vanishes on any pair of coordinate fields, because on such fields the components of $J$ are constant and all brackets vanish. Since $N$ is a tensor, it must vanish on every pair of vector fields. 
\end{proof}

 Observe that $(\Sigma,J_h)$ and $(\Sigma,J_{h_0})$ are biholomorphic if and only if the metrics $h$ and $h_0$ are conformal.

\subsubsection{Normalized period matrix} 
We follow Farkas and Kra \cite[pp. 56-64]{FK}. Fix a canonical homology basis $\chi=(\chi_1,\dots,\chi_{2g})$ of the Riemann surface $\Sigma$, and let $\omega$ be a holomorphic differential. Then the {\it period vector} of $\omega$ with respect to $\chi$ is
$$
(\chi,\omega)=\left(\oint_{\chi_1}\omega,\dots,\oint_{\chi_{2g}}\omega\right)\in \mathbb C^{2g}.
$$
Let now $\omega=(\omega_1,\dots,\omega_g)$ be a basis for the complex vector space $\mathcal H$ of holomorphic differentials. The {\it period matrix} of $\omega$ with respect to $\chi$ is the $g\times 2g$ matrix $(\chi,\omega)$ having rows $(\chi,\omega_1),\dots,(\chi,\omega_g)$. We can write it
$$
(\chi,\omega)=(B,C)
$$
where $B$ and $C$  are $g\times g$ matrices. It is standard that $B$ and $C$ are non-singular $g\times g$. By a suitable change of basis in $\mathcal H$  the right-hand side is written $(I,B^{-1}C)$. This fact, and the Riemann period relations, have the following consequence (see e.g., \cite[pp. 3-5]{martens_th})

\begin{thm} There is a unique basis $\zeta=(\zeta_1,\dots,\zeta_g)$ of $\mathcal H$ such that the period matrix of $\zeta$ is of the form
$$
(I,Z),
$$
where $Z=X+iY$ and $Y$ are real symmetric $g\times g$ matrices and $Y$ is positive definite. 
\end{thm}

\nero The matrix $(I,Z)$ is as above is called the {\it normalized period matrix}.
It is uniquely determined by the choice of the canonical homology basis. 
%


\medskip

Now assume that we choose another canonical homology basis, say $\tilde\chi=(\tilde\chi_1,\dots,\tilde\chi_{2g})$;  then $\tilde\chi=\chi M$ for a symplectic matrix $M\in{\rm Sp}_{2g}(\mathbb Z)$. 
Then:
$$
(I,Z)M=(B,C)
$$
for some invertible $B,C$, and the normalized period matrix with respect to $\tilde\chi$ and $(\tilde\zeta_1,\dots,\tilde\zeta_g)$ will be $(I,Z')$ with $Z'=B^{-1}C$. Therefore we say that 

\nero Two normalized period matrices $(I,Z)$ and $(I,Z')$ are said to be {\it symplectically equivalent} if there is a symplectic matrix $M$ such that $(I,Z)M=(B,C)$ and $B^{-1}C=Z'$. 

\subsubsection{Statement of Torelli's theorem  and conclusion of the proof} 

We state Torelli's Theorem in the following form (see \cite{martens_annals,martens_th}).

\begin{thm} Let $\Sigma$ and $\Sigma'$ be two compact Riemann surfaces of genus $g$ and let $(I,Z)$ and $(I,Z')$ be the normalized period matrices with respect to canonical homology bases $\chi$ of $\Sigma$ and $\chi'$ of $\Sigma'$, respectively. If $(I,Z)$ and $(I,Z')$ are symplectically equivalent, then $\Sigma$ and $\Sigma'$ are conformally equivalent.
\end{thm}

Now recall that there is a unique differentiable structure underlying any compact Riemann surface, and we can choose a common canonical homology basis $\chi$ for $\Sigma$ and $\Sigma'$. We conclude that if the normalized period matrices of $\Sigma$ and $\Sigma'$ with respect to $\chi$ coincide, then $\Sigma$ and $\Sigma'$ are conformally equivalent.

\medskip

In order to conclude the proof of Theorem \ref{secondmain} it remains to prove that the Gram matrix determines the normalized period matrix.


\medskip

We follow Farkas and Kra \cite[p. 56]{FK}. We fix a canonical homology basis $(\chi_1,\dots,\chi_{2g})$ with dual basis of harmonic forms $(\alpha_1,\dots,\alpha_{2g})$. The Gram matrix $\Gamma$ becomes a block matrix, each block being $g\times g$:
$$
\Gamma=\twomatrix ABCD.
$$
Next, we introduce the matrix $G$ representing the Hodge star operator $\star$ of $H$ with respect to the basis $(\alpha_1,\dots,\alpha_{2g})$. If
$
\mathcal A=(\alpha_1,\dots,\alpha_{2g})^t,
$
then
$$
\star\mathcal A=G\mathcal A.
$$
Further work shows that, if 
$
G=\twomatrix{\lambda_1}{\lambda_2}{\lambda_3}{\lambda_4}
$
then 
$
\Gamma=\twomatrix{\lambda_2}{-\lambda_1}{\lambda_4}{-\lambda_3}
$
and we have the identities:
$$
\lambda_1=-B, \quad \lambda_2=A, \quad \lambda_3=-D, \quad\lambda_4=B^t.
$$
Jumping to (2.8.1) one sees that the normalized period matrix $(I,Z)$ (in Farkas-Kra notation $Z$ is denoted $\Pi$) is such that
$$
Z=(-\lambda_3)^{-1}\lambda_1^t+i(-\lambda_3)^{-1}.
$$
By the previous identifications:
\begin{equation}\label{npm}
Z=-D^{-1}B^t+iD^{-1}.
\end{equation}
Then, the Gram matrix $\Gamma$ uniquely determines the normalized period matrix through the identities \eqref{npm}. Viceversa, if $Z=X+iY$ then using Riemann's period relations one sees that the Gram matrix is:
$$
\Gamma=\twomatrix{Y+XY^{-1}X}{-XY^{-1}}{-Y^{-1}X}{Y^{-1}}
$$
We summarize in the following theorem.

\begin{thm} The Gram matrix $\Gamma$ uniquely determines the normalized period matrix $Z$, and viceversa (both matrices being associated to the same canonical homology basis). 
\end{thm}

Therefore, by Torelli's theorem, if two Riemann surfaces have the same Gram matrix (with respect to a common canonical homology basis) then they are conformal. The proof of Theorem \ref{secondmain} is now complete.

\subsection{A simple proof for tori}
In the case of genus one surfaces, i.e., for tori, it is quite elementary to follow the lines of the proof of Theorem \ref{mainthm}. Indeed, the fact that two tori with the same ground state spectrum have the same volume and belong to the same conformal class follows by simple computations, which in turn allow a better understanding of the proof of Theorem \ref{mainthm}.

\medskip
Let $T$ be a genus one orientable surface, i.e., a torus. Consider $(T,h)$ where $h$ is a Riemannian metric on $T$. We will show how the Gram matrix $\Gamma=\Gamma(h)$ determines the conformal class of $h$. Recall that the Gram matrix is a conformal invariant. Let then $(T,\hat h)$ be the unique (up to homotheties) flat torus conformal to $(T,h)$. We can identify it with the square torus $\real 2/\mathbb Z^2$ endowed with the metric $\begin{pmatrix}a & b\\b & c\end{pmatrix}$, with $a,c>0$ and $ac-b^2>0$. We consider now the canonical homology basis $(\chi_1,\chi_2)$ corresponding to $e_1=(1,0)$ and $e_2=(0,1)$. The dual basis of harmonic forms is then  $(\alpha_1,\alpha_2)$,  with $\alpha_1=dx$, $\alpha_2=dy$. We can now compute the Gram matrix
$$
\Gamma=\begin{pmatrix}\|\alpha_1\|_{L^2}^2&\langle \alpha_1,\alpha_2\rangle_{L^2}\\
\langle \alpha_1,\alpha_2\rangle_{L^2} & \|\alpha_2\|_{L^2}^2 \end{pmatrix}=\frac{1}{\sqrt{ac-b^2}}\begin{pmatrix}c & {-b}\\-b&a\end{pmatrix}=\begin{pmatrix} A & B\\ B & D \end{pmatrix},
$$
where $A=\frac{c}{\sqrt{ac-b^2}}$, $B=\frac{-b}{\sqrt{ac-b^2}}$, $D=\frac{a}{\sqrt{ac-b^2}}$. Here we are using the same notation of \eqref{gram} .
As we have said, this is also the Gram matrix of $(T,h)$ (with respect to the same basis of harmonic $1$-forms), since $h$ and $\hat h$ are conformal.

\medskip

We observe immediately that ${\rm det}(\Gamma)=1$. It is trivial to see that if we know $\Gamma$, then we know $\hat h$ up to homotheties, and hence the conformal class of $(T,h)$. 

\medskip

We now calculate the normalized period matrix of $(T,h)$ for the genus $1$ case. 

\medskip

Given the canonical homology basis $(\chi_1,\chi_2)$, the intersection matrix $J$ is $J=\begin{pmatrix}0 & 1\\ {-1}& 0\end{pmatrix}$. Next we have to compute the matrix $G$ of $\star$ with respect to the basis $(\alpha_1,\alpha_2)$ of harmonic $1$-forms. To do so we compute
$$
\star \alpha_1=\frac{bdx+cdy}{\sqrt{ac-b^2}}\,,\ \ \ \star \alpha_2=\frac{-adx-bdy}{\sqrt{ac-b^2}}
$$
and hence 
$$
G=\frac{1}{\sqrt{ac-b^2}}\begin{pmatrix}b & c\\ -a & -b\end{pmatrix}=\begin{pmatrix}-B & A\\ -D & B\end{pmatrix},
$$
see \eqref{Gmatrix}.
We immediately see the relation $G=\Gamma J$, and the Riemann relations, again, become trivial.

\medskip We take now any holomorphic differential, for example, $\alpha=\alpha_1+i\star \alpha_1=dx+i\frac{bdx+cdy}{\sqrt{ac-b^2}}$, and compute its period matrix with respect to the fixed canonical homology basis $\chi=(\chi_1,\chi_2)$:
$$
(\chi,w)=\left(1+i\frac{b}{\sqrt{ac-b^2}},i\frac{c}{\sqrt{ac-b^2}}\right)=(1-iB,iA)\in \mathbb C^2.
$$
We know that there is a unique holomorphic differential $\zeta$ such that $(\chi,\zeta)=(1,X+iY)$ where $X,Y\in\mathbb R$, $Y>0$. We can compute it explicitly. In fact,
$$
(\chi,w)=\left(1+i\frac{b}{\sqrt{ac-b^2}},i\frac{c}{\sqrt{ac-b^2}}\right)=\eta\left(1,\frac{b}{a}+i\frac{\sqrt{ac-b^2}}{a}\right)=\eta\left(1,\frac{-B+i}{D}\right)\,,\ \ \ \eta=1+i\frac{b}{\sqrt{ac-b^2}}=1-iB.
$$
So the normalized period matrix is
$$
\left(1,\frac{b}{a}+i\frac{\sqrt{ac-b^2}}{a}\right)=\left(1,\frac{-B+i}{D}\right).
$$
and $\zeta=\eta w$. 

\medskip

Again it is straightforward to see that the conformal class $[h]$ determines the Gram matrix $\Gamma$, which in turn determines the normalized period matrix, and vice-versa. Note that we have used the uniformization theorem, which states that in a conformal class of metrics on $\Sigma$ there is a unique flat metric (up to homotheties).

\medskip

Now we go back to ground state isospectrality. If two metrics $h_1,h_2$ on $\Sigma$ have the same ground state spectrum, then
$$
\frac{\Gamma(h_1)}{|h_1|}=\frac{\Gamma(h_2)}{|h_2|}
$$
but since the Gram matrices have determinant $1$, this implies that $|h_1|=|h_2|$ hence the two metrics have the same volume and hence the same Gram matrices. Therefore they are conformal.

\section*{Acknowledgments}

The first author acknowledges support of the SNSF project “Geometric Spectral Theory”, grant
number 200020\_21257. The second  author acknowledges support of the project ``Perturbation problems and asymptotics for elliptic differential equations: variational and potential theoretic methods'' funded by the European Union – Next Generation EU and by MUR-PRIN-2022SENJZ. The first and second authors are members of the Gruppo Nazionale per le Strutture Algebriche, Geometriche e le loro Applicazioni (GNSAGA) of the Istituto Nazionale di Alta Matematica (INdAM).

\bibliographystyle{abbrv}
\bibliography{biblioCPSTori.bib}
\end{document}